\documentclass[12pt]{article}

\usepackage{amsmath}
\usepackage{amssymb}
\usepackage{amsthm}
\usepackage{fullpage}
\usepackage{verbatim}
\usepackage{stmaryrd}
\usepackage{authblk}
\usepackage{multirow}
\usepackage{mathrsfs}
\usepackage{longtable}
\usepackage{hyperref}
\usepackage{mathdots}
\usepackage{graphicx,subfigure}
\usepackage[english]{babel}
\usepackage[utf8]{inputenc}
\usepackage{mathtools}
\usepackage{chngcntr}
\usepackage{cleveref}
\usepackage{hhline}
\usepackage[utf8]{inputenc}
\usepackage{tikz,marvosym}
\usepackage[marvosym]{tikzsymbols}

\usetikzlibrary{decorations.pathmorphing,shapes,matrix}
\usetikzlibrary{patterns,cd,fit,quotes}
\usetikzlibrary{calc}
\usetikzlibrary{arrows}

\newcommand{\NEpath}[4]{
    \fill[white!25]  (#1) rectangle +(#2,#3);
    \fill[fill=white]
    (#1)
    \foreach \dir in {#4}{
        \ifnum\dir=0
        -- ++(1,0)
        \else
        -- ++(0,1)
        \fi
    } |- (#1);
    \draw[help lines] (#1) grid +(#2,#3);
    \draw[dashed] (#1) -- +(#3,#3);
    \coordinate (prev) at (#1);
    \foreach \dir in {#4}{
        \ifnum\dir=0
        \coordinate (dep) at (1,0);
        \else
        \coordinate (dep) at (0,1);
        \fi
        \draw[-,line width=2pt,blue] (prev) -- ++(dep) coordinate (prev);
    };
}

\newtheorem{Theorem}{Theorem}[section]
\newtheorem{Corollary}{Corollary}[section]
\newtheorem{Lemma}{Lemma}[section]
\newtheorem{Proposition}{Proposition}[section]
\newtheorem{Remark}{Remark}[section]

\newtheorem{Definition}[equation]{Definition}
\newtheorem{Definition-Remark}[equation]{Definition/Remark}
\newtheorem{Example}[equation]{Example}
\newtheorem{Notation}[equation]{Notation}

\numberwithin{equation}{section}

\newcommand{\C}{\mathbb{C}}
\newcommand{\N}{\mathbb{N}}

\newcommand{\Z}{\mathbb{Z}}

\newcommand{\MC}{\mathit{M}\kern -0.1em\mathrm{1}}
\newcommand{\MCC}{\mathit{M}\kern -0.1em\mathrm{2}}

\begin{document}

\title{Nearly Toric Schubert Varieties of Type A}

\author[1]{Mahir Bilen Can\thanks{The corresponding author.}}
\author[2]{Nestor D\'iaz Morera}

\affil[1]{{\small Tulane University, New Orleans; mahirbilencan@gmail.com}}    
\affil[2]{{\small Tulane University, New Orleans; ndiazmorera@tulane.edu}}    
\normalsize

\date{\today}
\maketitle

\begin{abstract}
A notion of a nearly toric variety is introduced. The examples of nearly toric varieties in the context of Schubert varieties are discussed. In particular, combinatorial characterizations of the smooth and singular nearly toric Schubert varieties are found. Furthermore, the counts and generating series are determined. Additionally, a connection between Dyck paths and a certain family of nearly toric Schubert varieties is established.\\
\medskip

\medskip

\end{abstract}

\section{Introduction}

Let $X$ be a $G$-variety, where $G$ is a connected reductive group. We fix a maximal torus $T$ and a Borel subgroup $B$ such that $T\subseteq B\subseteq G$. Let $c_T(X)$ (resp. $c_B(X)$) denote the minimum codimension of a $T$-orbit (resp. $B$-orbit) in $X$.  
In this article we consider a family of $G$-varieties that we call `nearly toric $G$-varieties,' characterized by the following three conditions: 
\begin{enumerate}
\item $X$ is a normal variety, 
\item $c_B(X) = 0$, and 
\item $c_T(X) = 1$.
\end{enumerate} 
The varieties possessing the first two properties are called spherical varieties, known for their remarkably rich geometric and representation-theoretic properties~\cite{Brion1993, VinbergKimelfeld}.
At the same time, varieties meeting the first and third conditions are interesting for their rich theory of divisors~\cite{AltmannPetersen, Timashev1997}.
In this article, we characterize the Schubert varieties satisfying all three of these properties.
We note that the first mentioned property is well-known: every Schubert variety is normal~\cite{Seshadri1984}.
Our purpose in this paper is twofold. First, we lay the combinatorial groundwork by clearly identifying the family of nearly toric Schubert varieties. Second, we establish both quantitative (enumerative) and qualitative (generators and relations) results concerning the singularities of these nearly toric varieties. 
We now proceed to explain our results in more detail.
\medskip

Hereafter, unless otherwise stated, $G$ will stand for the complex general linear group $\mathbf{GL}(n,\C)$, and $B$ 
will denote the Borel subgroup consisting of upper triangular matrices in $G$.
Finally, we will use the group of all diagonal matrices in $G$ as our maximal torus, $T$. 
Our examination reveals an intriguing relationship between the singularities of Schubert varieties and $c_T(X)$.
We found out that every singular Schubert variety $X$ with $c_T(X) = 1$ is a nearly toric variety. 
After making this discovery, it became clear to us that smooth and singular nearly toric Schubert varieties deserve distinct treatments, despite displaying interconnected features. For this reason, after establishing fundamental enumerative facts related to these two main families of nearly toric Schubert varieties, we shift our focus on a specific family of Schubert varieties. This particular family, indexed by the 312-avoiding permutations, exclusively comprises smooth Schubert varieties. We show that for these Schubert varieties, condition 3) implies condition 2). Hence, such Schubert varieties are nearly toric Schubert varieties as well.  
We now proceed to describe the results of this paper in more detail.
\medskip

Let $\mathbf{S}_n$ denote the symmetric group of $\{1,\dots, n\}$.
We identify $\mathbf{S}_n$ with the Weyl group of $G$.  
Let $Fl(n,\C)$ denote the variety of full flags of the vector space $\C^n$. 
In this notation, the first main result of our paper is as follows. 

\begin{Theorem}\label{intro:T1}
Let $X_w$, $w\in \mathbf{S}_n$, be a Schubert variety in the flag variety $Fl(n,\C)$. 
We assume that the $T$-complexity of $X_w$ is 1. 
Then the following statements hold: 
\begin{enumerate}
\item[(1)] $X_w$ is a singular nearly toric Schubert variety if and only if $w$ contains the pattern 3412 exactly once and avoids the pattern 321
(stated as Theorem~\ref{T:singulariff}).
\item[(2)] $X_w$ is a smooth nearly toric Schubert variety if and only if $w$ contains the 321 pattern exactly once and avoids 
both 3412  and 25314 patterns (stated as Theorem~\ref{T:smoothNRT}).
\end{enumerate}
\end{Theorem}

Lee, Masuda, and Park (\cite{LeeMasudaPark}) established permutation pattern criteria for $T$-complexity one Schubert varieties. However, they did not address spherical Schubert varieties within this class. A specific pattern avoidance criterion for identifying spherical Schubert varieties was only recently introduced (\cite{gao2021classification, Gaetz}). While the proof of Theorem~\ref{intro:T1} is relatively straightforward, it identifies all spherical Schubert varieties within the set of all $T$-complexity one Schubert varieties. Moreover, our theorem reveals an intriguing observation that we mentioned earlier: every singular $T$-complexity one Schubert variety is a spherical Schubert variety. We emphasize that the task of determining the nearly toric Schubert varieties in other types remains an open problem, both combinatorially and geometrically.
\medskip

Next, we find the counts for the two classes of nearly toric Schubert varieties. 

\begin{Theorem}\label{intro:T2}[stated as Theorem~\ref{T:rn} in the text]
Let $b_n$ denote the number of singular nearly toric Schubert varieties in $Fl(n,\C)$.
Let $r_n$ denote the number of smooth nearly toric Schubert varieties in $Fl(n,\C)$. 
If $F_m$ denotes the $m$-th Fibonacci number, then for every $n\geq 5$ we have 
\begin{enumerate}
\item[(1)] $b_n = \frac{2(2n-7)F_{2n-8} + (7n-23)F_{2n-7}}{5}$, and 
\item[(2)] $r_n = (n-2)F_{2n-4}$.
\end{enumerate}
\end{Theorem}

\medskip

Let $\mathbf{S}_n^{312}$ denote the set of all 312-avoiding permutations in $\mathbf{S}_n$. 
We call a Schubert variety $X_{w}$ such that $w\in \mathbf{S}_n^{312}$ a {\em partition Schubert variety}. 
These Schubert varieties are interesting for combinatorial as well as geometric reasons. 
They are rigid in the sense that their integral cohomology rings can be used for distinguishing amongst them~\cite[Theorem 1.1]{DMR}.
Our next main result is concerned with the spherical variety property of a partition Schubert variety.

\begin{Theorem}\label{intro:T3}[stated as Theorem~\ref{T:countofNTn312} in the text]
Let $w\in \mathbf{S}_n^{312}$. 
If $X_w$ is a $T$-complexity one Schubert variety, then $X_w$ is a spherical variety. 

Furthermore, if $\mathbf{NT}_n^{312}$ denotes the subset of elements $w\in \mathbf{S}_n^{312}$ such that $X_w$ is a nearly toric Schubert variety, then we have $|\mathbf{NT}_n^{312}|=(n-2)2^{n-3}$.
\end{Theorem}

The 312-avoiding permutations have interesting combinatorial interpretations.
For example, Bandlow and Killpatrick proved in~\cite{BandlowKillpatrick} that the set of all 312-avoiding permutations in $\mathbf{S}_n$ 
is in a bijection with the set of all Dyck paths of size $n$. 
Here, by a {\em Dyck path of size $n$} we mean a lattice path in $\Z^2$ that moves with unit North (0,1) and unit East (1,0) steps, starting at the origin
and ending at $(n,n)$ while staying weakly above the main diagonal. 
The Dyck paths play an important role in the theory of Macdonald polynomials,~\cite{Haglund}.
Now, let $N$ (resp. $E$) denote the unit North (resp. East) step. 
We call a lattice path in $\Z^2$ that is of the form $NN\dots N EE\dots E$, where the number of $N$'s and $E$'s are equal,
an {\em elbow}. 
Let $\pi$ be a Dyck path of size $n$. 
We call $\pi$ a {\em ledge} if its word is of the form 
\begin{align*}
\pi = \underbrace{NN\dots N}_{\text{$n-1$ North steps}} \underbrace{E\dots ENE\cdots EE}_{\text{$n$ East steps}},
\end{align*}
where the path ends with at least two East steps, and there is a unique isolated North step.
Finally, we call $\pi$ a {\em spherical Dyck path} if every ``connected component'' (see Definition~\ref{D:connectedcomponent}) $\tau$ of $\pi$ is either an elbow or a ledge,
or if every connected component of $\tau$ that lies above the line $y-x- 1=0$ is either an elbow or a ledge that extends to the main diagonal of $\pi$. 
For example, the Dyck paths in Figure~\ref{F:acceptable-intro} are spherical Dyck paths:

\begin{figure}[htp]
\begin{center}
\scalebox{0.7}{
\begin{tikzpicture}

\begin{scope}[xshift=-5.5cm]
\draw (0,0) grid (8,8);
\draw[color=blue, line width=3pt] (0,0) -- (0,3) -- (3,3) -- (3,6) -- (5,6) -- (5,7) --(7,7) --(7,8)-- (8,8);
\draw[ultra thick, dashed ] (0,0) -- (8,8);
\node at (4,-.75) {(a)};
\end{scope}

\begin{scope}[xshift=5.5cm]
\draw (0,0) grid (8,8);
\draw (0,0) grid (8,8);
\draw[color=blue, line width=3pt] (0,0) -- (0,2) -- (1,2) -- (1,5) -- (4,5) -- (4,8) -- (8,8);
\draw[ultra thick, dashed ] (0,0) -- (8,8);
\node at (4,-.75) {(b)};
\end{scope}
\end{tikzpicture}    
}
\end{center}
\caption{Two spherical Dyck paths.}
\label{F:acceptable-intro}
\end{figure}
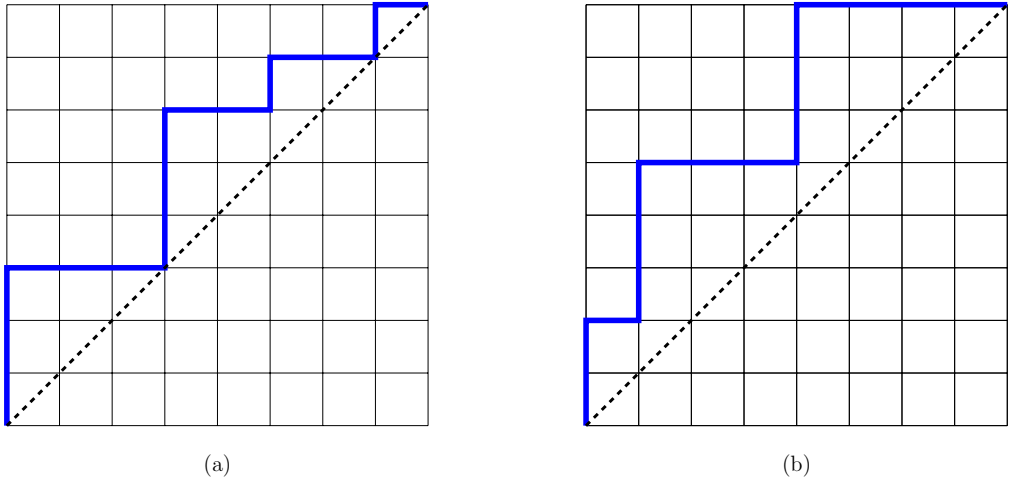
\medskip

We show that spherical Dyck paths correspond to spherical Schubert varieties $X_w$ if $w$ avoids 312 pattern. 
More precisely, we have the following statement. 

\begin{Theorem}\label{intro:T4}[stated as Theorem~\ref{T:sphericalDyckpaths} in the text]
Let $w\in \mathbf{S}_n^{312}$. 
Then $X_{w}$ is a spherical Schubert variety if and only if the Dyck path corresponding to $w$ is a spherical Dyck path.
\end{Theorem}

We structured our paper as follows. 
In Section~\ref{S:Preliminaries}, we setup our notation and provide the statements of some theorems that we use throughout our paper. 
The purpose of Section~\ref{S:25314} is to establish tools for analyzing the parametrizing set of smooth nearly toric Schubert varieties by advancing some results of Daly~\cite{Daly2010,Daly2013}. 
In Section~\ref{S:T1}, we prove the two parts of our Theorem~\ref{intro:T1}.
In Section~\ref{S:T2}, we prove our Theorem~\ref{intro:T2}. 
Additionally, we determine the generating series of the number of nearly toric Schubert varieties. 
The purpose of Section~\ref{S:Lattice} is to present some useful lemmas towards the proofs of our remaining results. 
In Section~\ref{S:Nearly}, we prove our Theorem~\ref{intro:T3}, determining the number of Dyck paths corresponding to the nearly toric partition Schubert varieties.
Finally, in Section~\ref{S:Allpartition}, we prove Theorem~\ref{intro:T4}.

\section{Preliminaries}\label{S:Preliminaries}

Throughout this article $\Z_+$ denotes the set of positive integers.
For $n\in \Z_+$, we set $[n]:=\{1,2,\dots, n\}$.

Let $G$ be a connected reductive algebraic group. 
Let $T$ and $B$ denote, respectively, a maximal torus and a Borel subgroup such that $T\subseteq B$. 
Let $W$ denote the Weyl group of $G$.
For $w\in W$, the Zariski closure $\overline{BwB/B} \subseteq G/B$, denoted $X_{w}$, is called the {\em Schubert variety} associated with $w$.

Let $L$ denote the {\it standard} Levi subgroup of the stabilizer group of $X_{w}$ in $G$. 
Here, {\it standard} refers to the assumption that $T\subseteq L$. 
We fix a Borel subgroup $B_L\subseteq L$ such that $T\subseteq B_L$.  
If $B_L$ has only finitely many orbits in $X_{w}$, then $X_{w}$ is said to be a {\em spherical Schubert variety}. 
More generally, a normal $G$-variety $X$ is called {\em spherical} if $B$ has an open orbit in $X$. 
It is a well-known result of Brion~\cite{Brion1986} and Vinberg~\cite{Vinberg1986} that a normal variety $X$ is a spherical $G$-variety if and only if $B$ has only finitely many orbits in $X$.
As we mentioned in the introduction, spherical varieties are characterized by a notion of a complexity as well.
Let $\C(X)$ denote the field of rational functions on $X$.
By a well-known result of Rosentlicht~\cite{Rosenlicht1956}, the $T$- and the $B$-complexities of $X$ are given by 
\begin{align*}
c_T(X) := \text{transcendence degree of the field of $T$-invariants $\C(X)^T$ over $\C$,}\\
c_B(X) := \text{transcendence degree of the field of $B$-invariants $\C(X)^B$ over $\C$.}
\end{align*}
In particular, $c_B(X) = 0$ means that $B$ has an open orbit. 
Therefore, a normal $G$-variety $X$ is spherical if and only if $c_B(T) = 0$.

In the rest of this article, we will use reductive groups of type A only. 
The complex general linear group of all invertible $n\times n$ matrices with entries from $\C$ is denoted by $\mathbf{GL}(n,\C)$. 
We denote by $\mathbf{B}(n,\C)$ the Borel subgroup of $\mathbf{GL}(n,\C)$ consisting of all upper triangular matrices. 
We denote by $\mathbf{T}(n,\C)$ the maximal torus consisting of all diagonal matrices in $\mathbf{B}(n,\C)$. 
To simplify our notation, we will use $G$, $B$, and $T$ to denote these groups. 
In this notation, the coset space $\mathbf{GL}(n,\C)/\mathbf{B}(n,\C)$ gives the complex flag variety $Fl(n,\C)$.

The Weyl group of $G$ is given by the symmetric group $\mathbf{S}_n$. 
The {\em Bruhat-Chevalley order on $\mathbf{S}_n$} is the order defined by 
\begin{align*}
v\leq w \iff X_{v}\subseteq X_{w} \qquad (v,w\in \mathbf{S}_n).
\end{align*}
It is well-known that $(\mathbf{S}_n,\leq)$ is a graded poset with the rank function, $\ell : \mathbf{S}_n\to \N,\ \ell(w) := \dim X_{w}$.
The number $\ell(w)$ is called the {\em length} of $w$. 
There are various equivalent descriptions of the length function $w\mapsto \ell(w)$. We mention some of them to establish notation. Further information about the length function and the Bruhat-Chevalley order can be found in reference~\cite[Chapter 2]{BjornerBrenti}.
Let $w_1w_2\dots w_n$ denote the {\em one-line expression} for $w$ so that $w_i$ is the value of $w$ at $i\in \{1,\dots, n\}$. 
Then a pair $(i,j)$, where $1\leq i < j \leq n$ is called an {\em inversion of $w$} if $w_j > w_i$. 
The number of inversions of $w$ equals $\ell(w)$. 
Let $s_i$ denote the simple transposition, $s_i:=(i, i+1)$, for $i\in [n-1]$. 
When confusion is unlikely, we denote the set $\{s_1,\dots, s_{n-1}\}$ by $S$. 
The length of an element $w\in \mathbf{S}_n$ is also given by the minimum number of simple transpositions required to write $w$ as a product. 
For $I\subseteq S$, we denote by $\mathbf{S}_I$ the parabolic subgroup of $\mathbf{S}_n$ that is generated by the elements of $I$. 
There is a unique maximal element of $\mathbf{S}_I$ with respect to Bruhat-Chevalley order.
We denote this element by $w_0(I)$.
However, if $I=S$, then we write $w_0$ instead of $w_0(I)$. 
For $w\in \mathbf{S}_n$, the {left descent set of $w$} is defined by $J(w) := \{ s \in S \mid \ell(sw) < \ell(w) \}$. 
Finally, an element $w\in \mathbf{S}_I$ is called a {Coxeter element of $\mathbf{S}_I$} if $w$ is given by a product of all simple transpositions from $I$ without repetition. 

We now go back to our discussion of spherical varieties. 
The question of determining which Schubert varieties are of complexity 1 is solved by Lee, Masuda, and Park in~\cite{LeeMasudaPark}. 
In~\cite[Theorem 1.2]{LeeMasudaPark} they show that, for $w\in \mathbf{S}_n$, the following statements are equivalent. 
\begin{enumerate}
\item $X_{w}$ is smooth and $c_T(X_{w})=1$. 
\item $w$ contains the pattern 321 exactly once and avoids $3412$.
\item  There exists a reduced word of $w$ containing $s_i s_{i+1} s_i$ as a factor and no other repetitions.
\end{enumerate}
In~\cite[Theorem 1.3]{LeeMasudaPark} the same authors show that for $w\in \mathbf{S}_n$, the following statements are equivalent. 
\begin{enumerate}
\item $X_{w}$ is singular and $c_T(X_{w})=1$. 
\item $w$ contains the pattern 3412 exactly once and avoids the pattern $321$.
\end{enumerate}
There are several useful criteria for checking which Schubert varieties are spherical.
For $w\in \mathbf{S}_n$, let $L$ denote the stabilizer group $L:=\mathrm{Stab}_G(X_w)$. 
The following assertion was conjectured in~\cite{HY2022}.
It has been recently proven in~\cite{CanSaha} and~\cite{GHY2024}:
\begin{itemize}
\item $X_w$ is a spherical $L$-variety if and only if $w_0(J(w))w$ is a Coxeter element of $\mathbf{S}_{J(w)}$.
\end{itemize}
The following assertion, referred to as the {\em pattern avoidance criterion of spherical Shubert varieties},
was initially conjectured in~\cite{gao2021classification} and subsequently proved in~\cite[Theorem 1.4]{Gaetz}.
\begin{itemize}
\item $X_w$ is a spherical $L$-variety if and only if $w$ avoids the following 21 patterns 
\begin{align}\label{A:Gaetz}
    \mathscr{P}:=\left\{ \begin{matrix}
    24531 & 25314 & 25341 &34512& 34521 & 35412 &35421 \\
     42531 & 45123 & 45213 & 45231 & 45312 & 52314 &52341\\
    53124 & 53142 & 53412 & 53421 & 54123 & 54213 & 54231
    \end{matrix} \right\}.
\end{align}
\end{itemize}
Finally, we want to mention the {\em Lakshmibai-Sandhya criterion for smoothness of Schubert varieties}, proved in~\cite{LakshmibaiSandhya}:
\begin{itemize}
\item $X_w$ is a smooth Schubert variety if and only if $w$ avoids the patterns $3412$ and $4231$.
\end{itemize}

As a direct corollary of these facts, we observe the following fact: 
\begin{Proposition}
A Schubert variety $X_w$ is a smooth spherical variety if and only if $w$ avoids the following nine patterns: 
\begin{align*}
    \mathscr{P}_s=\{
    24531, 25314, 34521, 35421, 53124, 54123, 54213 \} \cup \{3412, 4231\}.
\end{align*}
\end{Proposition}

\section{On the 25314 Pattern}\label{S:25314}

In the rest of the paper, for $i\in \N$, by $s_i$ we denote the $i$-th simple transposition, $s_i:=(i,\ i+1)$.
Let $\mathcal{B}_n$ denote the following subset of $\mathbf{S}_n$: 
\begin{align*}
\mathcal{B}_n &:= \left\{ w\in \mathbf{S}_n \mid \text{$w$ avoids 321, contains 3412 only once}\right\}.
\end{align*}
In~\cite[Proposition 5.2]{Daly2013}, Daly shows that a permutation $v\in \mathbf{S}_n$ is an element of $\mathcal{B}_n$ if 
there exists a reduced decomposition of $v$ having $s_i s_{i+1} s_{i-1} s_i$ as a factor with no other repetitions.
Let $\mathcal{M}_n$ denote the following subset of $\mathbf{S}_n$: 
\begin{align*}
\mathcal{A}_n &:= \left\{ w\in \mathbf{S}_n \mid \text{$w$ avoids 3412, contains 321 only once}\right\}.
\end{align*}
In~\cite[Theorem 4.10]{Daly2013}, Daly shows that $v\in \mathbf{S}_n$ is an element of $\mathcal{A}_n$ if 
there exists a reduced decomposition of $v$ having $s_i s_{i+1} s_i$ as a factor with no other repetitions.
Furthermore, in~\cite[Theorem 2.8]{Daly2010}, Daly shows that there is a bijection between $\mathcal{B}_{n+1}$ and $\mathcal{A}_n$.

We are interested in a special subset of $\mathcal{A}_n$.
\begin{align*}
\mathcal{M}_n &:= \left\{ w\in \mathbf{S}_n \mid \text{$w$ avoids 3412, contains 321 only once, and contains 25314} \right\}.
\end{align*}

\begin{Lemma}\label{L:existence1}
Let $v\in \mathbf{S}_n$. Then $v$ is an element of $\mathcal{M}_n$ if and only if there exits a reduced word $\underline{v}$ of $v$ containing $s_{i+2}s_{i+1}s_{i-1}s_{i}s_{i+1}$ for some $i\in \{2,\dots, n-3\}$ as a factor, and no simple transposition $s_j$ other than $s_{i+1}$ and $s_i$ appears more than once in $\underline{v}$. 
\end{Lemma} 

\begin{proof}
($\Leftarrow$)
Let $\underline{v}$ be a reduced word for $v$ such that $s_{i+2}s_{i+1}s_{i-1}s_{i}s_{i+1}$ appears as a factor and $\underline{v}$ contains no other repetitions of simple transpositions. 
Since $25314 = s_4 s_3 s_1s_2s_3$, the existence of the factor $s_{i+2}s_{i+1}s_{i-1}s_{i}s_{i+1}$ shows that $v$ contains the pattern 25312. 
We write $s_{i+1}s_{i-1}s_{i}s_{i+1}$ in the form $s_{i-1}s_{i+1}s_{i}s_{i+1}$. 
Hence, $s_{i+1}s_is_{i+1}$ appears in $v$. 
Since there is no repetition of other simple transpositions in $\underline{v}$, by~\cite[Theorem 4.10]{Daly2013}, 
$v$ avoids 3412 and contains 321 only once. 
Hence, $v$ is an element of $\mathcal{M}_n$.

($\Rightarrow$) 
We apply~\cite[Theorem 4.10]{Daly2013} once more.
Hence, there is a reduced word $\underline{v}$ having $s_i s_{i+1} s_i$ as a factor, for some $i\in \{1,\dots, n-1\}$, and $\underline{v}$ contains no other repetition 
of simple transpositions.  
We proceed to show that either $\underline{v}$ contains $s_{i+2}s_{i+1}s_{i-1}s_{i}s_{i+1}$ as a factor, or there is a reduced word that is obtained from $\underline{v}$ 
that contains $s_{i+2}s_{i+1}s_{i-1}s_{i}s_{i+1}$ as a factor. 
To this end, first, we observe that $25314 = s_4 s_3 s_1s_2s_3$. 
Evidently, the segment 531 in the pattern 25314 corresponds to the factor $s_{i+1} s_{i} s_{i+1}$ in $\underline{v}$.
Since 2 comes before 5314 in 25314, we see not only that $s_{i-1}$ appears in $\underline{v}$ but also that it appears to the left of $s_{i+1} s_{i} s_{i+1}$. 
Likewise, since 4 comes after 2531, we see not only that $s_{i+2}$ appears in $\underline{v}$ but also that it appears on the left of $s_{i+1} s_{i} s_{i+1}$.
Now, if we have $s_{i+2}s_{i-1}s_{i+1}s_is_{i+1}$ or $s_{i-1}s_{i+2}s_{i+1}s_is_{i+1}$, then we proved our claim. 
Otherwise, since there are no other repetitions in $\underline{v}$, we move every simple transposition that sits between the factors 
$s_{i+1}, s_{i-1}$, and $s_{i+1}s_is_{i+1}$ either to the right or to the left of them.
Hence, we obtain a new reduced word of $v$ having either $s_{i+2} s_{i-1} s_{i+1} s_{i} s_{i+1}$ or $s_{i-1} s_{i+2} s_{i+1} s_{i} s_{i+1}$ as a factor. 
Hence the proof of the implication $\Rightarrow$ is complete. 
\end{proof}

Let us assume that $w$ and $v$ are from $\mathcal{B}_n$.
Let $\underline{w}$ and $\underline{v}$ be reduced words for $w$ and $v$, respectively. 
In \cite[Proposition 2.2.]{Daly2010}, Daly proves that,
for $\{i,j\}\subseteq \{2,\dots, n-2\}$, if $s_{i}s_{i+1}s_{i-1}s_i$ and $s_{j}s_{j+1}s_{j-1}s_j$ are the factors of $\underline{w}$ and $\underline{v}$ respectively, 
then $w \neq v$ whenever $j \neq i$. 
We proceed to show that this feature holds for the elements of $\mathcal{M}_n$ as well.

\begin{Lemma}\label{L:uniqueness1}
Let $w$ and $v$ be two elements from $\mathcal{M}_n$.
Let $\underline{w}$ and $\underline{v}$ be reduced words for $w$ and $v$, respectively. 
We assume that there exists $\{i,j\}\subseteq \{2,\dots, n-3\}$ such that 
$s_{i+2}s_{i+1}s_{i-1}s_{i}s_{i+1}$ and $s_{j+2}s_{j+1}s_{j-1}s_js_{j+1}$ are factors of $\underline{w}$ and $\underline{v}$, respectively. 
Under these assumptions, if the indices $i$ and $j$ are not the same, then we have $w \neq v$. 
\end{Lemma}

\begin{proof}
Assume towards a contradiction that $w=v$. 
Then we have $\underline{w}s_i = \underline{v} s_i$.
Since $s_{i+2}s_{i+1}s_{i-1}s_{i}s_{i+1}$ is a factor of $\underline{w}$ and there is no other repetitions of simple transpositions in $\underline{w}$, 
we have $\ell(ws_i) = \ell(w) -1$. 
Let $W$ (resp. $V$) denote the set of simple transpositions that appear in the reduced word 
$\underline{w} := s_{t_1}s_{t_2}\cdots s_{t_r}$ (resp. $\underline{v} = s_{q_1}s_{q_2}\cdots s_{q_r}$).
Since $\underline{w}$ and $\underline{v}$ are two reduced words for the same permutation, we have $W=V$. 
At the same time, since $\ell(ws_i) = \ell(w) -1$, by the Exchange Property~\cite[Section 1.5]{BjornerBrenti}, we know that 
\begin{align}
w' &:= ws_i = s_{t_1}s_{t_2}\cdots \widehat{s}_{t_m} \cdots s_{t_r}, \label{A1}\\
v' &:= vs_i = s_{q_1}s_{q_2}\cdots \widehat{s}_{q_o} \cdots s_{q_r}. \label{A2} 
\end{align}
Here, the hats on the simple transpositions $s_{t_m}$ and $s_{q_o}$ indicate that we omit these simple transpositions from the reduced words $\underline{w}$ and $\underline{v}$. 
While initially we are not certain about the relationship between $s_{q_o}$ and $s_i$, we are certain that $s_{t_m} = s_i$ holds. 
This assertion arises from the presence of the factor $s_{i+1}s_i s_{i+1}$ in $\underline{w}$ without any other repetitions of simple transpositions. 
At the same time, considering that the right-hand sides of (\ref{A1}) and (\ref{A2}) are reduced words, the sets of simple transpositions in these expressions are identical.
Given that $V=W$, we conclude that the reduced word $\underline{v}$ must contain either $s_i s_{i+1} s_i$ or $s_{i+1}s_i s_{i+1}$ as a factor. 
Otherwise, the set of simple transpositions in (\ref{A2}) would be a proper subset of $W$, which contradicts our earlier equality.
In either scenario, where $\underline{v}$ has one of the specified factors, the absence of repeated simple transpositions in $\underline{v}$ leads us to the conclusion that the factor $s_{j+2}s_{j+1}s_{j-1}s_js_{j+1}$ in $\underline{v}$ must be equal to the factor $s_{i+2}s_{i+1}s_{i-1}s_{i}s_{i+1}$. This, in turn, implies that $s_i = s_j$, thereby accomplishing the proof we set out to establish.
\end{proof}

As a corollary of these lemmas, we have the following proposition.

\begin{Proposition}\label{P:uniquereducedword}
Let $v \in \mathbf{S}_n$. 
Then $v$ is an element of $\mathcal{M}_n$ if and only if there exits a reduced word $\underline{v}$ of $v$ containing $s_{i+2}s_{i+1}s_{i-1}s_{i}s_{i+1}$ for some $i\in \{2,\dots, n-3\}$ as a factor, and no other repetitions. 
\end{Proposition}

\begin{proof}
The existence of the factor $s_{i+2}s_{i+1}s_{i-1}s_{i}s_{i+1}$ is proved in Lemma~\ref{L:existence1}.
The uniqueness of such an expression is proved in Lemma~\ref{L:uniqueness1}.
Hence, the proof of our proposition follows. 
\end{proof}

We proceed to construct a map, 
\begin{align*}
\Psi: \mathcal{M}_{n+1} \rightarrow \mathcal{B}_n.
\end{align*}
Let $v\in \mathcal{M}_{n+1}$. 
By Proposition~\ref{P:uniquereducedword}, $v$ has a unique reduced word $\underline{v}$ containing the word $s_{i+2}s_{i+1}s_{i-1}s_{i}s_{i+1}$ 
(for some $i\in \{2,\dots, n-3\}$) as a factor and $\underline{v}$ has no other repeated simple transpositions. 
Let us write $\underline{v}$ as follows: 
\begin{align*}
\underline{v} = s_{j_1} s_{j_2 } \cdots s_{j_m} s_{j_{m+1}} s_{j_{m+2}} s_{j_{m+3}} s_{j_{m+4}} \cdots s_{j_r},
\end{align*}
where $s_{j_m} s_{j_{m+1}} s_{j_{m+2}} s_{j_{m+3}} s_{j_{m+4}} = s_{i+2}s_{i+1}s_{i-1}s_{i}s_{i+1}$. 
We have essentially three distinct cases to consider:
\begin{enumerate}
\item[(1)] $s_{i+2} = s_n$,
\item[(2)] $s_{i+2} \neq s_n$ but $s_n$ is an element of $\{s_{j_1},\dots, s_{j_r}\} \setminus \{ s_{j_{m}},\dots, s_{j_{m+4}}\}$, 
\item[(3)] $s_n$ is not an element of $\{s_{j_1},\dots, s_{j_r}\}$.
\end{enumerate}
In each of these three cases, we will construct an appropriate element of $\mathcal{B}_n$. 
\medskip

\noindent Case (1). 

In this case, we have $s_{j_m} s_{j_{m+1}} s_{j_{m+2}} s_{j_{m+3}} s_{j_{m+4}} = s_ns_{n-1}s_{n-3}s_{n-2}s_{n-1}$.
Then, we set 
\begin{align}\label{A:case1}
\Psi(v) := s_{j_1} s_{j_2 } \cdots s_{j_{m-1}} \underbrace{s_{j_m} s_{j_{m+1}} s_{j_{m+2}} s_{j_{m+3}} s_{j_{m+4}}}_{\text{replaced with 
$s_{n-2} s_{n-1} s_{n-3} s_{n-2}$}} s_{j_{m+5}} \cdots s_{j_r}.
\end{align}
Here, since we removed the only occurrence of $s_n$ from $\underline{v}$, we obtain an element of $\mathbf{S}_n$. 
Also, since $s_{n-2} s_{n-1} s_{n-3} s_{n-2}$ appears as a factor of $\Psi(v)$, and no other simple transposition has a repetition in $\Psi(v)$, 
we see that $\Psi(v)\in \mathcal{B}_n$. 
\medskip

\noindent Case (2). 

Assume that $s_n$ is the $d$-th factor $s_{j_d}$ in $\underline{v}$. 
Without loss of generality we assume that $m+4<d$. 
Now, we first define $\underline{w}_0$ as in (\ref{A:case1}), and then replace $s_{j_d}$ in $\underline{w}_0$ by $s_{i+2}$ to define our element $\underline{w}:=\Psi(v)$:
\begin{align}\label{A:case2}
\Psi(v) := s_{j_1} s_{j_2 } \cdots s_{j_{m-1}} \underbrace{s_{j_m} s_{j_{m+1}} s_{j_{m+2}} s_{j_{m+3}} s_{j_{m+4}}}_{\text{replaced with $s_{i} s_{i+1} s_{i-1} s_{i}$}}
 s_{j_{m+5}}\cdots \underbrace{s_{j_d}}_{s_{i+2}} \cdots s_{j_r}.
\end{align}
Notice that since $s_n$ does not appear in $\underline{w}$, we have $\underline{w} \in \mathbf{S}_n$. 
Furthermore, $s_{i} s_{i+1} s_{i-1} s_{i}$ appears as a factor of $\underline{w}$ and no other simple transposition has a repetition in $\underline{w}$.
Hence, we see that $\underline{w}\in \mathcal{B}_n$. 
\medskip

\noindent Case (3). 

This is similar to case (1).  
We replace the factor $s_{i+2}s_{i+1}s_{i-1}s_{i}s_{i+1}$ in $\underline{v}$ by $s_{i} s_{i+1} s_{i-1} s_{i}$, 
\begin{align}\label{A:case3}
\Psi(v) := s_{j_1} s_{j_2 } \cdots s_{j_{m-1}} \underbrace{s_{j_m} s_{j_{m+1}} s_{j_{m+2}} s_{j_{m+3}} s_{j_{m+4}}}_{s_{i} s_{i+1} s_{i-1} s_{i}}
 s_{j_{m+5}} \cdots s_{j_r}.
\end{align}
Since $s_n$ is not a factor of $\underline{v}$, $\Psi(v)$ is automatically an element of $\mathbf{S}_n$. 
At the same time, $s_{i} s_{i+1} s_{i-1} s_{i}$ is a factor of $\Psi(v)$ and no simple transposition is repeated.
Hence, we have $\Psi(v)\in \mathcal{B}_n$. 
\medskip

\begin{Theorem}\label{T:RecurrenceML}
The map $\Psi$ defined above is a bijection between $\mathcal{M}_{n+1}$ and $\mathcal{B}_n$.
\end{Theorem}

\begin{proof}
The fact that $\Psi$ is well-defined is obvious from the construction. 
We proceed to show that $\Psi$ is a bijection.

We begin with the proof of the injective property of $\Psi$. 
Let $v=s_{i_1}\cdots s_{i_r}$ and $v' = s_{i_1'}\cdots s_{i_r'}$ be two elements from $\mathcal{M}_{n+1}$ such that 
\begin{align*}
\Psi(v) = \Psi(v') = s_{t_1}s_{t_2}\cdots s_{t_{r-1}}.
\end{align*}
We will show that $v=v'$.
If both $v$ and $v'$ are as in (1), then it is clearly true that every factor of $v$ and $v'$ are the same. 
In particular, we have $v=v'$. 
The proofs of the following two cases are similar: either both $v$ and $v'$ are as in (2) or both of them are as in (3), 
then $\Psi(v)=\Psi(v')$ implies $v=v'$. 

Let us assume that $v$ is as in (1) and $v'$ is as in (2). 
We will show that this case can not occur. 
Notice that if $v$ is as in (1), then we have $i+2=n$. 
Hence, there exists $m\in \{1,\dots, r-3\}$ such that the factor $s_{t_m}s_{t_{m+1}} s_{t_{m+2}} s_{t_{m+3}}$ of $\Psi(v)$ is given by 
$s_{t_m}s_{t_{m+1}} s_{t_{m+2}} s_{t_{m+3}}= s_{n-2}s_{n-1}s_{n-3}s_{n-2}$.
But since $v'$ is as in (2), we know that the segment $s_{i_m'}s_{i_{m+1}'} s_{i_{m+2}'} s_{i_{m+3}'} s_{i_{m+4}'}$ in $\underline{v'}$ does not 
contain $s_n$. In particular, we know that $s_{i_{m+2}'} \neq s_{n-1}$. 
This means that the product $s_{n-2}s_{n-1}s_{n-3}s_{n-2}$ is not a factor of $\Psi(v')$ to begin with. 
In other words, $\Psi(v')$ cannot be equal to $\Psi(v)$ to begin with.

Let us assume that $v$ is as in (1) and $v'$ is as in (3). 
We will show that this case can not occur as well.  
As before, since $v$ is as in (1), we have $i+3=n$. 
Hence, there exists $m\in \{1,\dots, r-3\}$ such that the factor $s_{t_m}s_{t_{m+1}} s_{t_{m+2}} s_{t_{m+3}}$ of $\Psi(v)$ is given by $s_{t_m}s_{t_{m+1}} s_{t_{m+2}} s_{t_{m+3}}= s_{n-2}s_{n-1}s_{n-3}s_{n-2}$.
Since $v'$ is as in (3), we know that the segment $s_{i_m'}s_{i_{m+1}'} s_{i_{m+2}'} s_{i_{m+3}'} s_{i_{m+4}'}$ in $\underline{v'}$ does not 
contain $s_n$. 
Hence, the factor $s_{i_m'}s_{i_{m+1}'} s_{i_{m+2}'} s_{i_{m+3}'} s_{i_{m+4}'}$ of $\underline{v}'$ is replaced by the factor $s_{i_{m+1}'} s_{i_m'} s_{i_{m+3}'}s_{i_{m+4}'}$
to get $\Psi(v')$, we see that $s_{i_{m+1}'} =s_{n-1}$. 
But this implies that $s_{i_m'} = s_n$, contradicting the fact that $v'$ is as in (3). 
It follows that $\Psi(v')= \Psi(v)$ is not possible for such $v$ and $v'$.

The remaining case is when $v$ is as in (2) and $v'$ is as in (3).
In this case, in order for $\Psi(v)$ to be equal to $\Psi(v')$, the factor $s_{i_d}'$ in $\underline{v}'$ must be equal to $s_{i+2}$ to begin with.
But this is not possible since $v'$ is as in (3), hence, the factor $s_{t_m'}s_{t_{m+1}'} s_{t_{m+2}'} s_{t_{m+3'}} s_{t_{m+4}'}$ of $\underline{v}'$ 
is given by $s_{i+2}s_{i+1}s_{i-1}s_{i}s_{i+1}$, contradicting with our no-repetition assumption on the elements of $\mathcal{M}_{n+1}$.

We proceed to show that $\Psi$ is surjective. 

Let $s_{t_1}s_{t_2}\cdots s_{t_{r-1}}$ be an element of $\mathcal{B}_n$. 
Then there exists $m\in \{1,\dots, r-4\}$ such that $s_{t_m}s_{t_{m+1}}s_{t_{m+2}}s_{t_{m+3}} = s_i s_{i+1} s_{i-1} s_i$ for some $i\in \{2,\dots, n-2\}$.
We have three cases. 

If $i=n-2$, then we define $v:=s_{j_1}\cdots s_{j_r}$ by setting 
\begin{align*}
s_{j_x}:=
\begin{cases}
s_{t_x} & \text{ for $x\in \{1,\dots, m-1\}$}, \\
s_{t_{x-1}} & \text{ for $x\in \{m+5,\dots, r\}$},
\end{cases}
\end{align*}
and $s_{j_m}s_{j_{m+1}}s_{j_{m+2}}s_{j_{m+3}} s_{j_{m+4}} := s_ns_{n-1}s_{n-3}s_{n-2}s_{n-1}$.
Clearly, this element $\underline{v}$ is an element of $\mathcal{M}_{n+1}$ as in (1) and $\Psi(v) = s_{t_1}s_{t_2}\cdots s_{t_{r-1}}$.

Next, let us assume that $i< n-2$ and $s_{i+2}=s_{t_d}$ for some $d\in \{1,\dots, r-1\} \setminus \{m,m+1,m+2,m+3\}$.
Then we define $v:=s_{j_1}\cdots s_{j_r}\in \mathcal{M}_{n+1}$ as follows.
First, we replace $s_{t_d}$ with $s_n$ in $s_{t_1}s_{t_2}\cdots s_{t_{r-1}}$.
Then we replace the factor $s_{t_m}s_{t_{m+1}}s_{t_{m+2}}s_{t_{m+3}}$ with $s_{i+2} s_{i+1}s_{i-1}s_{i}s_{i+1}$.
It is easy to check that $v$ is an element of $\mathcal{M}_{n+1}$ as in (2) and that $\Psi(v) = s_{t_1}s_{t_2}\cdots s_{t_{r-1}}$.

Finally, let us assume that $i<n-2$ and $s_{i+2}$ does not appear in $\{s_{t_1},\dots, s_{t_{r-1}}\}$. 
Then we define $v:=s_{j_1}\cdots s_{j_r}\in \mathcal{M}_{n+1}$ by replacing the factor $s_{t_m}s_{t_{m+1}}s_{t_{m+2}}s_{t_{m+3}}$ with $s_{i+2} s_{i+1}s_{i-1}s_{i}s_{i+1}$.
It is straightforward to check that $v$ is an element of $\mathcal{M}_{n+1}$ as in (3) and that $\Psi(v) = s_{t_1}s_{t_2}\cdots s_{t_{r-1}}$.

Hence, the map $\Psi$ is onto.
This finishes the proof of the fact that $\Psi$ is a bijection.
\end{proof}

\begin{Example}
Let $v$ denote the permutation given by $v=2\,6\,3\,1\,4\,7\,8\,5$ in one-line notation. 
It is easy to see that $v$ is an element of $\mathcal{M}_8$. 
In particular, there is a reduced word $\underline{v}$ of $v$ containing $s_{i+2}s_{i+1}s_{i-1}s_is_{i+1}$ for some $i \in \{2,3,4,5\}$ by Proposition~\ref{P:uniquereducedword}.
For example, if $i=2$, then the set of all reduced words of $v$ containing $s_4s_3s_1s_2s_3$ as a subexpression is given by 
\begin{align*}
\{ s_5\textcolor{red}{s_4s_3s_1s_2s_3}s_6s_7, s_5s_6\textcolor{red}{s_4s_3s_1s_2s_3}s_7, s_5s_6s_7\textcolor{red}{s_4s_3s_1s_2s_3}\}.
\end{align*}
Let us work with the following element:
\begin{align*}
\underline{v}:=s_5s_6s_7\textcolor{red}{s_4s_3s_1s_2s_3} = s_{j_1}s_{j_2}s_{j_3}\textcolor{red}{s_{j_4}s_{j_5}s_{j_6}s_{j_7}s_{j_8}}.
\end{align*}
Since $s_{j_4}\neq s_7$ and $s_7$ is in $\{s_{j_1},...,s_{j_8}\}\setminus \{s_{j_4},...,s_{j_8}\}$, the image of $\underline{v}$ under the map $\Psi$ is found by applying Case (2) as in (\ref{A:case2}),
\begin{align*}
\Psi(\underline{v})=s_5s_6s_4 \textcolor{blue}{s_2s_3s_1s_2}.
\end{align*} 
More specifically, we find that $\Psi(\underline{v})=3\,6\,1\,2\,4\,7\,5 \in \mathcal{B}_7$.
\end{Example}

\begin{Corollary}\label{C:RecurrenceML}
Let $n\geq 3$. 
Then we have $| \mathcal{A}_n | = |\mathcal{M}_{n+2}|$. 
\end{Corollary}
\begin{proof}
The proof is a direct consequence of Theorem~\ref{T:RecurrenceML} and Daly's result~\cite[Theorem 2.8]{Daly2010} that $|\mathcal{B}_{n+1}| = |\mathcal{A}_n|$.
\end{proof}

\section{Nearly Toric Schubert Varieties and Singularities}\label{S:T1}

This section is about the relationship between smoothness and sphericalness properties of Schubert varieties with $T$-complexity one.
In particular, we prove the two parts of our Theorem~\ref{intro:T1} from the introduction section. 

\begin{Lemma}\label{L:cT=1cB=0}
Let $X_w$ be a singular Schubert variety in $Fl(n,\C)$. 
If $c_T(X_w) = 1$, then $c_B(X_w) = 0$. In other words, the singular Schubert varieties of $T$-complexity one are spherical.
\end{Lemma}

\begin{proof} 
   In Figure~\ref{F:fig:3}, we list the permutations $v\in \mathscr{P}$ that contains the pattern 321 at least once.
   This means that if $w\in \mathbf{S}_n$ avoids the pattern 321, then it will avoid $v$. 
   \begin{figure}[htp]
    \centering
    \begin{tabular}{|c|c|c|c|c|c|c|} \hhline{|=|=|=|=|=|=|=|}
    24531 & 25314  & 25341 &34521 &35421 &42531 &45213 \\\hline
         431& 531 &531 &321 &321 &421 &421 \\\hhline{|=|=|=|=|=|=|=|} 
   45231 & 45312 & 52314 & 52341 & 53124 &53142 &53412 \\\hline
421 &431 & 521 & 521 & 531 &531 &531  \\ \hhline{|=|=|=|=|=|=|=|}
 53421& 54123& 54213 &54231 \\ \hhline{|-|-|-|-|}
  532& 541  &542 &542 \\ \hhline{|=|=|=|=|}
    \end{tabular}
    \caption{The occurrences of the pattern $321$ in the elements of $\mathscr{P}$.}
   \label{F:fig:3}
\end{figure}
   In Figure~\ref{F:fig:4}, we look at the elements $u\in \mathscr{P}$ that do not appear in Figure~\ref{F:fig:3}.
   It turns out that all of these permutations contain the pattern 3412 at least twice. 
\begin{figure}[htp]
    \centering
    \begin{tabular}{|c|c|c|}\hhline{|=|=|=|}
        34512 &  35412 & 45123 \\\hhline{|=|=|=|}
        3412 & 3412 &4512 \\
         4512  & 3512  &4523 \\\hline
    \end{tabular}
    \caption{The pattern $3412$ appears at least twice.}
    \label{F:fig:4}
\end{figure}

It follows from these observations and~\cite[Theorem 1.2]{LeeMasudaPark}, which is stated in Section~\ref{S:Preliminaries}, 
 that if a permutation contains the pattern 3412 exactly once and avoids the pattern 321,
then it avoids all of the elements of $\mathscr{P}$.
Hence, our assertion follows from the pattern avoidance criterion of spherical Schubert varieties.
\end{proof}

\begin{Example}\label{ex:0.4.3}
Let $X_{w}$ be a singular nearly toric Schubert variety in $Fl(5,\C)$.
Then $w$ is an element of the following set:
\begin{align*}
 \left\{ \begin{matrix}
12543, & 13542, & 14325, &14352, & 21543, & 23541, \\
24315, & 24351, & 32145, & 32154, & 32415, & 32451
    \end{matrix} \right\}.
\end{align*}
\end{Example}

We are ready to prove the first part of our first theorem from the introduction section. 
\begin{Theorem}\label{T:singulariff}
Let $n\geq 1$. 
Let $X_w$ be a Schubert variety in $Fl(n,\C)$ such that $c_T(X_w) = 1$. 
Then $X_w$ is a singular spherical Schubert variety if and only if $w$ contains the pattern 3412 exactly once and avoids the pattern 321. 
\end{Theorem}

\begin{proof}
If $X_w$ is a singular $T$-complexity one Schubert variety, then our claim follows from ~\cite[Theorem 1.3]{LeeMasudaPark} that is mentioned in the preliminaries section. Conversely, if $w$ contains the pattern 3412, then $X_w$ is singular by the Lakshmibai-Sandhya criterion for smoothness~\cite{LakshmibaiSandhya}. 
Then Lemma~\ref{L:cT=1cB=0} shows that $X_w$ is spherical. This finishes the proof of our theorem. 
\end{proof}

Let $X_w$ be Schubert variety in $Fl(n,\C)$. 
We assume that $c_T(X_w) = 1$ and $X_w$ is smooth.
Our goal is to determine when $X_w$ is a spherical variety. 

\begin{Theorem}\label{T:smoothNRT}
A smooth Schubert variety $X_{w}$ such that $c_T(X_w)=1$ is non-spherical if and only if 
$w$ is $3412$-avoiding, contains the pattern $321$ exactly once, and it contains the pattern $25314$. 
\end{Theorem}

\begin{proof}
Let $X_w$ be a smooth Schubert variety such that $c_T(X_w) = 1$. 
Then we know that 
\begin{enumerate}
\item[(1)] $w$ avoids both 3412 and 4231 patterns, 
\item[(2)] $w$ contains exactly one 321 pattern. 
\end{enumerate}
Comparing it with the patterns that determine the spherical variety property of $X_w$, we see 
that the first condition (1) automatically implies that $w$ avoids the box-indicated entries from the following array of patterns: 
\begin{align*}
  \begin{bmatrix}
    24531 & 25314 &  \boxed{25341} &  \boxed{34512} & 34521 &  \boxed{35412} &35421 \\
    \boxed{42531} &  \boxed{45123} &  \boxed{45213} &  \boxed{45231} &  \boxed{45312} &  \boxed{52314} &  \boxed{52341}\\
    53124 &  \boxed{53142} &  \boxed{53412} &  \boxed{53421} & 54123 & 54213 &  \boxed{54231}
    \end{bmatrix}
\end{align*}
In other words, a smooth Schubert variety $X_w$ is spherical if and only if $w$ avoids the remaining patterns which we included in the following set:
\begin{align*}
    \{ 24531,  25314, 34521, 35421, 53124,  54123, 54213 \}.
\end{align*}
But each of these seven patterns, except 25314, has more than one occurrences of 321 pattern in them. 
In other words, if $w$ satisfies conditions (1) and (2) above, then $w$ automatically avoids the box-indicated patterns from the following set:
\begin{align*}
    \{  \boxed{24531},  25314,  \boxed{34521},  \boxed{35421},  \boxed{53124},   \boxed{54123},  \boxed{54213} \}.
\end{align*}
Hence, a smooth Schubert variety $X_w$ which is of $T$-complexity 1 is non-spherical if and only if it contains the pattern 25314.
This finishes the proof of our theorem.
\end{proof}

The proof of the following corollary directly follows from the definition of $\mathcal{M}_n$.

\begin{Corollary}
The set $\mathcal{M}_n$ parametrizes the smooth Schubert varieties $X_w$ in $Fl(n,\C)$ such that $c_T(X_w)=1$ and $X_w$ is non-spherical. 
\end{Corollary}

\section{Enumeration of Nearly Toric Schubert Varieties}\label{S:T2}

In this section, we solve the problem of enumerating nearly toric Schubert varieties.

Let $d_n$ denote the cardinality of $\mathcal{M}_n$. 
Let $a_n$ (resp. $b_n$) denote the cardinality of $\mathcal{A}_n$ (resp. $\mathcal{B}_n$). 
In~\cite[Theorem 3.7]{Daly2010}, Daly shows that the generating series for $a_n$ is given by 
\begin{align*}
\sum_{n\geq 3} a_n x^n = \frac{x^3}{(1-3x+x^2)^2}.
\end{align*}
Since $a_n = b_{n+1}$, we see at once that 
\begin{align*}
\sum_{n\geq 4} b_{n} x^{n} = \sum_{n\geq 3} b_{n+1} x^{n+1} = \sum_{n\geq 3} a_n x^{n+1} =  \frac{x^4}{(1-3x+x^2)^2}.
\end{align*}
Likewise, since $a_n = d_{n+2}$, we see at once that 
\begin{align*}
\sum_{n\geq 5} d_{n} x^{n} = \sum_{n\geq 3} c_{n+2} x^{n+2} = \sum_{n\geq 3} a_n x^{n+2} =  \frac{x^5}{(1-3x+x^2)^2}.
\end{align*}

By combining the results of Daly, Lee, Park, and Masuda, we see that 
\begin{align*}
\left|\left\{w\in \mathbf{S}_n \mid X_w\subseteq Fl(n,\C)\ \text{ such that $c_T(X_w) = 1$}\right\}\right| = a_n+b_n. 
\end{align*}
Let us determine the generating series of $T$-complexity one Schubert varieties, 
\begin{align*}
\sum_{n\geq 4} (a_n+b_n) x^n  = \sum_{n\geq 4} a_n x^n + \sum_{n\geq 4} b_n x^n &= \sum_{n\geq 4} b_{n+1} x^n + \sum_{n\geq 4} b_n x^n \\ 
&= \frac{1}{x} \sum_{n\geq 4} b_{n+1} x^{n+1} + \sum_{n\geq 4} b_n x^n \\ 
&= \frac{1}{x}\left(-b_4x^4+ \sum_{n\geq 4} b_{n} x^{n}\right) + \sum_{n\geq 4} b_n x^n \\ 
&= \frac{1}{x}\left(-x^4+ \sum_{n\geq 4} b_{n} x^{n}\right) + \sum_{n\geq 4} b_n x^n \\ 
&= -x^3 + \frac{x+1}{x} \sum_{n\geq 4} b_{n} x^{n}\\
&= -x^3 +  \frac{x^4+x^3}{(1-3x+x^2)^2}
\end{align*}

Now, the number of smooth $T$-complexity one Schubert varieties in $Fl(n,\C)$ is given by $a_n$. 
The number of smooth $T$-complexity one Schubert varieties $X_w$ in $Fl(n,\C)$ such that $X_w$ is not spherical is given by $d_n$. 
Therefore, the number of smooth nearly toric Schubert varieties in $Fl(n,\C)$ is given by $a_n - d_n$. 
Equivalently, if we denote by $r_n$ the number of smooth nearly toric Schubert varieties in $Fl(n,\C)$ by $r_n$, then we have 
\begin{align*}
r_n = a_n - d_n = a_n - a_{n-2}.
\end{align*}
for $n\geq 5$. 
The generating series for the smooth nearly toric Schubert varieties is found as follows:
\begin{align*}
\sum_{n\geq 5} r_n x^n = \sum_{n\geq 5} (a_n - a_{n-2})x^n &= \sum_{n\geq 5} a_nx^n - x^2 \sum_{n\geq 3} a_n x^n \\
&= -a_3x^3 -a_4x^4+ \frac{x^3}{(1-3x+x^2)^2} - x^2 \left( \frac{x^3}{(1-3x+x^2)^2} \right).
\end{align*}
For $n=3$, we have six Schubert varieties in $Fl(3,\C)$.
All but one of them is a smooth toric variety.
The full flag variety $Fl(3,\C)$ itself is not a toric variety. 
At the same time, since every Schubert divisor in $Fl(3,\C)$ is a toric variety, we see that the $T$-complexity of $Fl(3,\C)$ is one. 
In other words, $Fl(3,\C)$ is the unique smooth nearly toric Schubert variety in $Fl(3,\C)$. 
Therefore, we have $a_3= 1$. 
In $Fl(4,\C)$, every Schubert variety is spherical. 
This was checked in~\cite{can2020sphericality} but it also follows easily from the pattern avoidance criterion for spherical Schubert varieties.  
At the same time, by using \cite[Theorem 1.2]{LeeMasudaPark} that we mentioned in the preliminaries section, we see that $a_4 = 6$.
In conclusion, we have 
\begin{align*}
\sum_{n\geq 5} r_n x^n =  -x^3 - 6x^4+ \frac{x^3-x^5}{(1-3x+x^2)^2}. 
\end{align*}

We record this as a proposition.
\begin{Proposition}\label{P:powerseriesforr_n}
The generating series for the number of smooth nearly toric Schubert varieties in $Fl(n,\C)$, $n=1,2,\dots$ is given by 
\begin{align*}
\sum_{n\geq 5} r_n x^n =  -x^3 - 6x^4+ \frac{x^3-x^5}{(1-3x+x^2)^2}. 
\end{align*}
\end{Proposition}

Recall that the Fibonacci numbers are defined by the recurrence $F_n = F_{n-1}+F_{n-2}$ and the initial conditions
$F_0 = F_1 = 1$.
We are now ready to prove our second theorem from the introduction. 
For convenience of the reader, we state it here. 

\begin{Theorem}\label{T:rn}
Let $F_k$ denote the $k$-th Fibonacci number. 
The number of smooth nearly toric Schubert varieties in $Fl(n,\C)$, where $n\geq 5$, is given by 
\begin{align*}
r_{n} = (n-2) F_{2(n-2)}.
\end{align*}
The number of singular nearly toric Schubert varieties in $Fl(n,\C)$ is given by 
\begin{align}\label{A:bn}
b_n = \frac{2(2n-7)F_{2n-8} + (7n-23)F_{2n-7}}{5}.
\end{align}
\end{Theorem}

\begin{proof}
In~\cite{Daly2010}, Daly shows that 
\begin{align}\label{A:Dalysformulaforan}
a_n = \frac{2(2n-5)F_{2n-6} + (7n-16)F_{2n-5}}{5}. 
\end{align}
Since $b_n = a_{n-1}$, our formula (\ref{A:bn}) follows.
Next, we will calculate $r_{n+2}$ by using the formula, $r_{n+2} = a_{n+2}-a_n$.
We apply the recurrence for the Fiboncci numbers repeatedly:  
\begin{align*} 
r_{n+2} &= a_{n+2} - a_n \\
&= \frac{ 2(2(n+2)-5)F_{2(n+2)-6} + (7(n+2)-16)F_{2(n+2)-5} -2(2n-5)F_{2n-6} - (7n-16)F_{2n-5} }{5} \\
&= \frac{ (4n-2)F_{2n-2} + (7n-2)F_{2n-1} -(4n-10)F_{2n-6} - (7n-16)F_{2n-5} }{5} \\ 
&= \frac{ (4n-2)F_{2n-2} + (7n-2) (F_{2n}-F_{2n-2})  -(4n-10)(F_{2n-4}-F_{2n-5}) -(7n-16)F_{2n-5} }{5} \\ 
&= \frac{ (7n-2) F_{2n}- 3nF_{2n-2}  -(4n-10)F_{2n-4}-(3n-6)F_{2n-5} }{5} \\ 
&= \frac{ (7n-2) F_{2n}- 3nF_{2n-2}  -(4n-10)F_{2n-4}-(3n-6) (F_{2n-3}-F_{2n-4}) }{5} \\ 
&= \frac{ (7n-2) F_{2n}- 3nF_{2n-2} - (3n-6) F_{2n-3} - (n-4) F_{2n-4}) }{5} \\ 
&= \frac{ (7n-2) F_{2n}- 3nF_{2n-2} - (3n-6) F_{2n-3} - (n-4) (F_{2n-2}-F_{2n-3})) }{5} \\ 
&= \frac{ (7n-2) F_{2n}- (4n-4) F_{2n-2} - (2n-2) F_{2n-3} }{5} \\ 
&= \frac{ (7n-2) F_{2n}- (2n-2)F_{2n-1} - (2n-2) F_{2n-2} }{5} \\ 
&= \frac{ (7n-2) F_{2n}- (2n-2)F_{2n} }{5}\\
&=nF_{2n}. 
\end{align*}
This finishes the proof.
\end{proof}

\begin{Theorem}
For $n\geq 0$, let $t_n$ denote the number of nearly toric Schubert varieties in $Fl(n,\C)$. 
Then the generating series of $t_n$ is given by 
\begin{align*}
\sum_{n\geq 0} t_nx^n =   \frac{x^3+x^4-x^5}{(1-3x+x^2)^2}.
\end{align*}
Furthermore, we have 
\begin{enumerate}
\item[(1)] $t_0 = t_1 = t_2 = 0$, $t_3 =1$, $t_4=7$.
\item[(2)] For $n\geq 5$, $t_n$ is given by the sum
\begin{align*}
t_n = r_n+a_{n-1} = (n-2)F_{2n-2} + \frac{2(2n-7)F_{2n-8} + (7n-23)F_{2n-7}}{5}.
\end{align*}
\end{enumerate}
\end{Theorem}

\begin{proof}
Notice that $t_n = r_n + b_n$ for $n\geq 2$. 
For $n\in \{0,1,2\}$, every Schubert variety is a smooth toric variety. 
Indeed, $Fl(0,\C)$ is a point, $Fl(1,\C)$ is the projective line $\mathbb{P}^1$, 
and $Fl(2,\C)$ is the projective plane $\mathbb{P}^2$. 
It follows that $t_0=t_1=t_2 = 0$. 
We already observed previously that $t_3 = a_3= 1$. 
Since $a_n = b_{n+1}$, we see that $b_4=1$. 
Also, we found already that $a_4=6$. 
Hence, $t_4 = a_4+b_4= 7$. 
It remains to calculate $\sum_{n\geq 5} t_nx^n$. 
For this computation we use Proposition~\ref{P:powerseriesforr_n}:
\begin{align*}
\sum_{n\geq 5} (r_n + b_n)x^n &= \sum_{n\geq 5} r_n +\sum_{n\geq 5}  b_n\\
&= \left( -x^3 - 6x^4+  \frac{x^3-x^5}{(1-3x+x^2)^2}\right) + \left( - b_4x^4 + \sum_{n\geq 4} b_{n} x^{n}\right) \\
 &= -x^3 - 6x^4+ \frac{x^3-x^5}{(1-3x+x^2)^2} - x^4 +  \frac{x^4}{(1-3x+x^2)^2}  \\
  &= -x^3 - 7x^4+ \frac{x^3+x^4-x^5}{(1-3x+x^2)^2}. \\
\end{align*}
Combining this with $t_3x^3 + t_4x^4 = x^3 + 7x^4$ we finish the proof of our first assertion. 
For our second assertion, we already know that $t_0,\dots,t_4$. 
For the second part, we combine Theorem~\ref{T:rn} and (\ref{A:Dalysformulaforan}).
This finishes the proof of our theorem.
\end{proof}

\section{Lattice Paths}\label{S:Lattice}

We begin with recalling the ingredients of our combinatorics. 
By a {\it lattice path in $\Z^2$} we mean a sequence 
$\pi:=((x_0,y_0),\dots, (x_k,y_k)) \in (\Z^2)^{k+1}$ such that for each $i\in \{1,\dots, k\}$, 
we have $(x_i,y_i)\in \{ (x_{i-1}+1,y_i), (x_{i},y_{i-1}+1)\}$.
In this case, $(x_0,y_0)$ (respectively, $(x_k,y_k)$) is called the starting point (respectively, the ending point) of $\pi$. 
It is our implicit assumption that when we add the unit length line segments between the consecutive lattice points of $\pi$ we get a connected path 
in the euclidean topology of the plane. 

It will be advantageous to view a lattice path $((x_0,y_0),\dots, (x_k,y_k)) \in (\Z^2)^{k+1}$ as a word, $a_1a_2\dots a_k$, where for each $i\in \{1,\dots, k\}$, the letter $a_i$ is defined by 
\begin{align*}
a_i := 
\begin{cases}
N & \text{ if $(x_i,y_i) = (x_i,y_{i-1}+1)$},\\
E & \text{ if $(x_i,y_i) = (x_{i-1}+1,y_i)$}.
\end{cases}
\end{align*}
Here, $N$ and $E$ stand for a \emph{unit North} and a \emph{unit East} step, respectively. 
For $n \in \mathbb{Z}_+$, a \emph{Dyck path of size $n$} is a lattice path $\pi$ such that the number of $N$'s in any initial segment of the word representation of $\pi$ is greater than or equal to the number of $E$'s in the same initial segment, and the total number of $N$'s as well as the total number of $E$'s in $\pi$ is equal to $n$. 
Although it is not necessary, as a convention, we assume that a Dyck path of size $n$ starts at $(0,0)$ and ends at $(n,n)$.
\begin{Notation}
The set of all Dyck paths of size $n$ will be denoted by $L_{n,n}^+$. 
\end{Notation}
It is well-known that~\cite{BandlowKillpatrick} the set of all 312-avoiding permutations in $W$ is in a bijection with the set of all Dyck paths of size $n$. 
We call a Dyck path $\pi \in L_{n,n}^+$ whose word is given by 
\[
\underbrace{NN\dots N}_{\text{$n$ North steps }} \underbrace{EE\dots E}_{\text{ $n$ East steps}}
\] 
an {\em elbow}. 
It is called a {\em ledge} if its word is of the form 
\[
\pi = \underbrace{NN\dots N}_{\text{$n-1$ North steps}} \underbrace{E\dots ENE\cdots EE}_{\text{$n$ East steps}},
\]
where there are at least two East steps at the end of the path, and there is a unique isolated North step.
In other words, a Dyck path $\pi$ in $L_{n,n}^+$ is called a {\em ledge} if its word starts with $n-1$ North steps followed by an East step,
and ends with at least two East steps. 
\medskip

Let $\pi \in L_{n,n}^+$. 
A lattice point $p$ on $\pi$ is called a {\em peak} if $p$ is the highest point of a North step and the left most point of an East step of $\pi$. 
In other words, a lattice point $p$ on $\pi$ is called a peak if a moving particle on $\pi$, moving towards $(n,n)$, arrives at $p$ by a North step, and then departs from it by an East step. 
The {\em primary dip} of $\pi$ is the last diagonal point that $\pi$ hits before reaching $(n,n)$.
The {\em secondary dip} of $\pi$ is the last lattice point of the form $(b-1,b)$, where $0<b<n$, that $\pi$ hits before reaching $(n,n)$. 
In Figure~\ref{F:fig0}, we depict a Dyck path of size $10$ (on the left).
On the right hand side of the same figure, we marked the peaks of the Dyck path by the solid thick black dots.  
The circles indicate the primary and the secondary dips of the Dyck path.
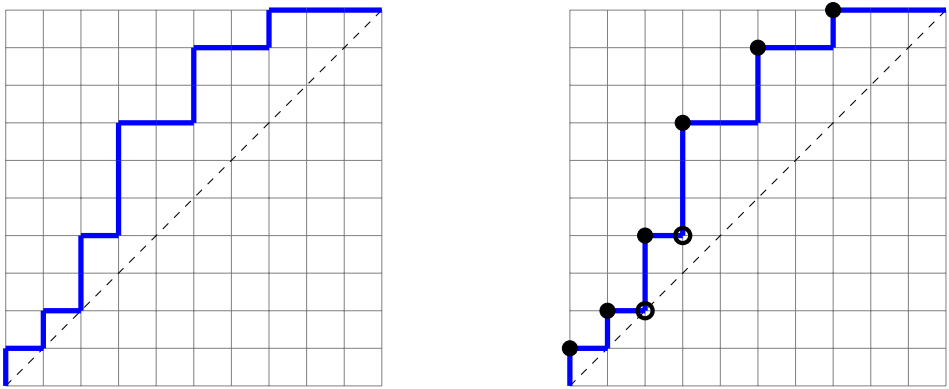
\begin{figure}[htp]
\begin{center}
 \begin{tikzpicture}[scale=0.5]
 \begin{scope}[xshift=-7.5cm]
    \NEpath{0,0}{10}{10}{1,0,1,0,1,1,0,1,1,1,0,0,1,1,0,0,1,0,0,0};
    \end{scope}
    
     \begin{scope}[xshift=7.5cm]
     \NEpath{0,0}{10}{10}{1,0,1,0,1,1,0,1,1,1,0,0,1,1,0,0,1,0,0,0};
    \draw[fill=black] (0,1) circle   (.2cm);
      \draw[fill=black] (1,2) circle   (.2cm);
          \draw[fill=black] (2,4) circle   (.2cm);
             \draw[fill=black] (3,7) circle   (.2cm);
              \draw[fill=black] (5,9) circle   (.2cm);
              \draw[fill=black] (7,10) circle   (.2cm);
                  \draw[ultra thick] (3,4) circle   (.2cm);
                  \draw[ultra thick] (2,2) circle   (.2cm);
    \end{scope}
    
\end{tikzpicture}    
\end{center}
\caption{A Dyck path of size 10 with marked peaks and dips.}
\label{F:fig0}
\end{figure}

For $r\in \{0,1,\dots, n\}$, the line defined by the equation $y-x-r=0$ in the $xy$-plane will be called the {\em $r$-th diagonal}.
The 0-th diagonal is called the {\em main diagonal}.
For $\pi \in L_{n,n}^+$, let $t_\pi$ be the number defined by
\begin{align*}
t_\pi:= \max \{ s \in \{1,\dots, n\} :\ \text{there is a peak of $\pi$ on the $s$-th diagonal}\}.
\end{align*}
Notice that for each $r\in \{1,\dots, t_\pi\}$ the portion of $\pi$ that lies weakly above the $r$-th diagonal, denoted by $\pi^{(r)}$, 
gives a family of ``shifted'' Dyck paths. 
We call $\pi^{(r)}$ the {\em $r$-th subpath system of $\pi$}.
For example, in Figure~\ref{F:fig1}, we depict an element $\pi\in L_{8,8}^+$ along with its 1-st and 2-nd subpath systems. 
\begin{figure}[htp]
\begin{center}
 \begin{tikzpicture}[scale=0.55]
 
\begin{scope}[xshift= -11cm]
\NEpath{0,0}{8}{8}{1,1,0,0,1,1,1,1,0,0,0,1,1,0,0,0}
\node at (4,-1) {$\pi$};
\end{scope}

\begin{scope}[xshift= 0cm]
\NEpath{0,0}{8}{8}{}
\draw[dashed] (0,1) -- (7,8);
\draw[ultra thick, color=blue] (0,1) -- (0,2) -- (1,2);
\draw[ultra thick, color=blue] (2,3) -- (2,6) -- (5,6) -- (5,8) -- (7,8);
\node at (4,-1) {$\pi^{(1)}$};
\end{scope}

\begin{scope}[xshift= 11cm]
\NEpath{0,0}{8}{8}{}
\draw[dashed] (0,2) -- (6,8);
\draw[dashed] (0,1) -- (7,8);
\draw[fill=blue] (0,2) circle   (.2cm);
\draw[ultra thick, color=blue] (2,4) -- (2,6)-- (4,6);
\draw[ultra thick, color=blue] (5,7) -- (5,8)-- (6,8);
\node at (4,-1) {$\pi^{(2)}$};
\end{scope}
\end{tikzpicture}
\end{center}
\caption{A Dyck path and its subpath systems.}
\label{F:fig1}
\end{figure}
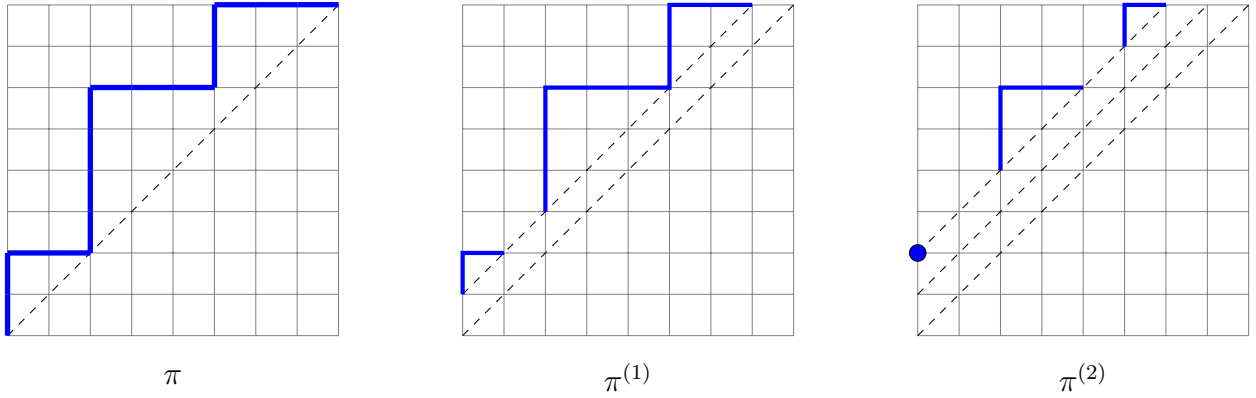

\subsection{312-avoiding permutations and Dyck paths.}

The {\em area statistic} on $L_{n,n}^+$ is the function $\text{area}: L_{n,n}^+\to \N$ defined by 
\[
\text{area}(\pi):=\text{the number of unit squares between $\pi$ and the 0-th diagonal,}
\]
where $\pi\in L_{n,n}^+$.

There is a natural distributive lattice structure on $L_{n,n}^+$.
It is defined as follows. 
Let $\pi$ and $\tau$ be two Dyck paths from $L_{n,n}^+$. 
Then we have 
\[
\pi \preceq \tau \iff \text{$\pi$ lies weakly below $\tau$}.
\]
Note that $(L_{n,n}^+,\preceq)$ is a graded lattice with the rank function $\text{area} : L_{n,n}^+ \to \N$.

It is a well-known result of Bandlow and Killpatrick,~\cite{BandlowKillpatrick} that there is a bijection
\begin{align}\label{A:BKbijection}
\psi : \mathbf{S}_n^{312} &\longrightarrow L_{n,n}^+,
\end{align}
such that $\ell (w) = \text{area}(\psi(w))$.
In~\cite{BBFP2005}, it is shown that $\psi$ is in fact a poset isomorphism between $(\mathbf{S}_n^{312},\leq)$ and $(L_{n,n}^+,\preceq)$. 
We refer to $\psi$ as the {\em Bandlow-Killpatrick isomorphism}. 
Since it will be useful for our purposes, we will review on an example the inverse of the Bandlow-Killpatrick isomorphism.

\begin{Example}\label{E:BK}

We begin with the Dyck path $\pi \in L_{8,8}^+$ that is depicted on the left hand side of Figure~\ref{F:pathandperm}. 

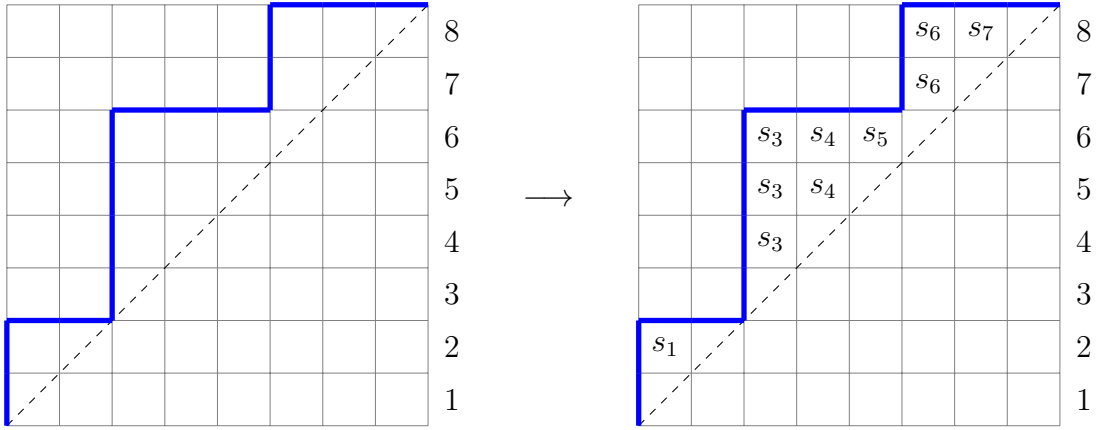
\begin{figure}[htp]
\begin{center}
 \begin{tikzpicture}
 
 \begin{scope}[scale=0.7, xshift = -6cm]
\NEpath{0,0}{8}{8}{1,1,0,0,1,1,1,1,0,0,0,1,1,0,0,0}
 \node[right=2pt] at (8,.5) {1};
  \node[right=2pt] at (8,1.5) {2};
   \node[right=2pt] at (8,2.5) {3};
    \node[right=2pt] at (8,3.5) {4};
     \node[right=2pt] at (8,4.5) {5};
      \node[right=2pt] at (8,5.5) {6};
       \node[right=2pt] at (8,6.5) {7};
        \node[right=2pt] at (8,7.5) {8};
         \end{scope}
         
         \begin{scope}[xshift=3cm]
          \node at (0,3) {$\longrightarrow$};
         \end{scope}
         
          \begin{scope}[scale=0.7, xshift = 6cm]
\NEpath{0,0}{8}{8}{1,1,0,0,1,1,1,1,0,0,0,1,1,0,0,0}
 \node[right=2pt] at (8,.5) {1};
  \node[right=2pt] at (8,1.5) {2};
   \node[right=2pt] at (8,2.5) {3};
    \node[right=2pt] at (8,3.5) {4};
     \node[right=2pt] at (8,4.5) {5};
      \node[right=2pt] at (8,5.5) {6};
       \node[right=2pt] at (8,6.5) {7};
        \node[right=2pt] at (8,7.5) {8};
        
         \node at (.5,1.5) {$s_1$};
         \node at (2.5,3.5) {$s_3$};
          \node at (2.5,4.5) {$s_3$};
           \node at (2.5,5.5) {$s_3$};
           
            \node at (3.5,4.5) {$s_4$};
          \node at (3.5,5.5) {$s_4$};
          
           \node at (4.5,5.5) {$s_5$};
           \node at (5.5,6.5) {$s_6$};
           
           \node at (5.5,7.5) {$s_6$};
           \node at (6.5,7.5) {$s_7$};
                 \end{scope}

\end{tikzpicture}    
\end{center}
\caption{The construction of a permutation from a Dyck path.}
\label{F:pathandperm}
\end{figure}

We fill the full squares between $\pi$ and the $0$-th diagonal by simple transpositions as follows: 
\begin{enumerate}
\item The first full square to the left of the main diagonal on the $i$-th row gets the simple transposition $s_{i-1}$. 
\item The indices of the simple transpositions decrease one-by-one from right to left along the rows. 
\item The full squares on the left hand side of $\pi$ are not filled.
\end{enumerate}
Next, we read the filling by the simple transpositions from left to right, top to bottom.
We use bars between two simple transpositions only when we move to a new row while reading. 
The product of simple transpositions along a row will be called a {\em segment}.  
If a row does not contain any simple transpositions, then we write $\emptyset$ for the corresponding segment. 
Thus, the reading word of the filling of $\pi$ on the right hand side of Figure~\ref{F:pathandperm} is given by 
\begin{align}\label{A:barredpermutations}
w:=s_6 s_7 | s_6 | s_3 s_4 s_5 | s_3 s_4 | s_3|\emptyset | s_1 | \emptyset.
\end{align}
Now we will view this expression as a permutation of $\{1,\dots, 8\}$. 
However, before proceeding, it is crucial to emphasize an important point. 
According to the convention used in~\cite{BandlowKillpatrick}, the values of $w$ are determined by evaluating the simple transpositions from left to right. For instance, to determine the value of $w$ at 4, we first apply $s_6$, then $s_7$, followed by $s_6$, and then $s_3$, and so forth. Consequently, the one-line expression of $w$ is as follows:

\begin{align*}
w = 2\ 1 \ 6 \ 5 \ 4 \ 8 \ 7 \ 3.
\end{align*}
\end{Example}

We wish to reiterate the convention employed in this context. In this method of representing the one-line expression $w_1\ldots w_n$ for an element $w\in \mathbf{S}_n$, the left descents correspond to the simple transpositions $s_i$ ($i\in [n-1]$) where $w_i > w_{i+1}$. For example, the left descent set of $w = 2\ 1 \ 6 \ 5 \ 4 \ 8 \ 7 \ 3$ consists of $\{ s_1,s_3,s_4,s_6, s_7\}$.

\begin{Notation}
Hereafter, we denote by $\phi: L_{n,n}^+ \to \mathbf{S}_n^{312}$ the inverse of the Bandlow-Killpatrick isomorphism $\psi : \mathbf{S}_n^{312}
\to L_{n,n}^+$.
\end{Notation}

\begin{Remark}\label{R:itisreduced}
Let $\pi \in L_{n,n}^+$. 
Let us write $w_\pi$ for the permutation $\phi(\pi)$.
Since $\ell(w_\pi) = \text{area}(\pi)$, we see that the product of the simple transpositions of the segments 
$\sigma_1,\sigma_2 , \cdots, \sigma_n$ (multiplied in the reading order, from left to right)
gives a reduced expression for $w_\pi$. 
By abuse of notation, we denote this reduced expression of $w_\pi$ by 
$\sigma_1 \sigma_2   \cdots  \sigma_n$ although some of the segments here might be empty.
\end{Remark}

We now proceed to prove several lemmas that we will use in the sequel.

\begin{Lemma}\label{L:Helpful1}
Let $\pi \in L_{n,n}^+$. 
Let $w_\pi$ denote $\phi(\pi)$.
Let $w_1w_2\ldots w_n$ be the one-line expression of $w_\pi$. 
If $(a,a)$ is the primary dip of $\pi$, then we have $w_n = a+1$.
\end{Lemma}

\begin{proof}
We begin with writing $w_\pi$ in segments. 
\[
\phi(\pi) = w_\pi  = \sigma_1 | \sigma_2 | \cdots | \sigma_{n},
\]
where $\sigma_i$, for $i\in \{1,\dots, n\}$, is the reading word of the $n-i+1$-th row of the simple transposition filled diagram of $\pi$.
Note that $\sigma_i$ might be vacuous. 
In that case, we move to $\sigma_{i+1}$ while reading the simple transpositions from left to right.

To find the value $w_n$, we apply $\sigma_1 | \sigma_2 | \cdots | \sigma_{n}$ to $n$ starting from the left most simple transposition
in $\sigma_1$. 
A moment of thought reveals that the smallest index $i\in \{1,\dots, n\}$ such that $\sigma_i =\emptyset$ is determined by the coordinates 
of the first dip of $\pi$. 
More precisely, we have
\[
\min \{ i :\ \sigma_i=\emptyset \} = n-a. 
\]
If $\sigma_1 = \emptyset$, then we see at once that $a=n-1$. 
In this case, the simple transposition $s_{n-1}$ does not occur in the reading word of $\pi$. 
In other words, we have $w_\pi (n) = n$. 
Hence, our claim follows.

We proceed with the assumption that $\sigma_1$ is nonempty. 
Let us denote $n-a-1$ by $k$. 
Then we have $k\geq 1$ and $\sigma_{k+1}=\emptyset$. 
Also, we have $\sigma_j\neq \emptyset$ for $j\in \{1,\dots, k\}$. 
The segments $\sigma_1,\dots,\sigma_k$ end with the simple transpositions $s_{n-1},s_{n-2},\dots, s_{n-k}$, respectively. 
When we apply $\sigma_1$ to $n$, we get $n-1$. 
When we apply $\sigma_2$ to $n-1$, we get $n-2$.
Continuing in this manner, when we apply $\sigma_{n-k}$ to $n-k+1$ we get $n-k$.  
Note that $n-k=a+1$. 
Thus, we see that if we apply $\sigma_1 \sigma_2\cdots \sigma_k$ to $n$ on the left, then we get $a+1$. 
We now observe that the segments $\sigma_{k+2}\sigma_{k+3}\dots \sigma_n$ do not have any simple transposition $s_j$
such that $j > a-1$. 
Therefore, the application of $\sigma_{k+2}\sigma_{k+3}\dots \sigma_n$ to $a+1$ on the left gives $a+1$.
Hence, the proof of our claim is finished. 
\end{proof}

\begin{Lemma}\label{L:Helpful2}
Let $\pi \in L_{n,n}^+$. 
Let $w_\pi$ denote $\phi(\pi)$.
Let $w_1w_2\ldots w_n$ be the one-line expression of $w_\pi$. 
Let $(a,a)$ (resp. $(b-1,b)$) be the primary (resp. secondary) dip of $\pi$.
If $a\leq n-2$, then we have $(w_{n-1}-1,w_{n-1})=(b,b+1)$. 
\end{Lemma}

\begin{proof}
The proof of this lemma is very similar to the proof of Lemma~\ref{L:Helpful1}.
We will follow the notation that we introduced in that proof. 
In particular, $k$ is given by $k=n-a$.
Since $a\leq n-2$, we know that $\sigma_1\neq \emptyset$. 
We proceed with the assumption that $\sigma_1 = s_{n-1}$. 
In this case, if we apply the segments $\sigma_1\dots \sigma_k$ to $n-1$ on the left, then we get $n$.
Since $\sigma_{k+1},\dots, \sigma_n$ do not contain the simple transposition $s_{n-1}$, 
after applying $\sigma_{k+1}\cdots \sigma_n$ to $n$, we see that $w_\pi(n-1) = w_{n-1} = n$.  
Hence, in this case, we have $(w_{n-1}-1,w_{n-1}) = (n-1,n)$. 
By its definition, this is the secondary dip of $\pi$. Hence our claim follows in this case. 

We proceed with the assumption that $\sigma_1$ is of the form $\sigma_1 = s_{t_1} s_{t_1+1}\cdots s_{n-1}$
for some $t_1\in \{1,\dots, n-2\}$. 
Then when we apply $\sigma_1$ to $n-1$ we obtain $n-2$. 
Since $\sigma_1$ has more than one simple transpositions in it, $\sigma_2$ has to have at least one simple transposition in it.
Hence, we know that it ends with $s_{n-2}$. 
Depending on whether $s_{n-3}$ appears in $\sigma_2$ or not, when we apply $\sigma_2$ to $n-2$ on the left, we obtain either $n-2$ or $n-1$. 
In fact, the latter case can occur if and only if $\sigma_2 = s_{n-2}$. 
Then we find our secondary dip at $(n-2,n-1)$. 
At the same time, since the rest of the segments of $w_\pi$ do not contain $s_{n-2}$, we find that $w_\pi(n-1) = n-1$. 
Thus, in this case, our claim follows. 
We now proceed with the assumption that $\sigma_2$ is of the form $\sigma_2 = s_{t_2} s_{t_2+1}\cdots s_{n-2}$
for some $t_2\in \{1,\dots, n-3\}$. 
But this means that our procedure will continue inductively until we reach $\sigma_{k-1}$.
Without loss of generality we may assume that $\sigma_{k-1}$ is given by $\sigma_{k-1} =s_{a}s_{a+1}$. 
Then, by applying the reasoning we used earlier, we find that $w_{n-1} = a+2$, and that $(a+1,a+2)$ is the secondary dip of $\pi$. 
This finishes the proof of our assertion. 
\end{proof}

\begin{Lemma}\label{L:Helpful3}
Let $\pi\in L_{n,n}^+$. 
Let $w_\pi$ denote $\phi(\pi)$.
Let $w_1w_2\ldots w_n$ be the one-line expression of $w_\pi$. 
If $t$ denotes the number of East steps of $\pi$ on the line $y=n$, then we have $w_{n-t+1}=n$.
\end{Lemma}

\begin{proof}
Let $t$ denote the number of East steps of $\pi$ on the line $y=n$.
If $t=1$, then our claim follows at once since in this case we have $w_n = n$.
We proceed with the assumption that $t> 1$. 
Then we know that  
\begin{align*}
\sigma_1 = s_{n-t+1}  s_{n-t} s_{n-t-1}\cdots s_{n-1},
\end{align*}
which shows at once that $\sigma_1(n-t+1)=n$. 
Since the rest of the segments of $w_\pi$ do not contain the simple transposition $s_{n-1}$, we see that $w_\pi(n-t+1)=n$.
This finishes the proof of our assertion. 
\end{proof}

\section{Nearly Toric Partition Schubert Varieties}\label{S:Nearly}

Our goal in this section is to identify the set of Dyck paths corresponding to the set of nearly toric partition Schubert varieties.

\begin{Theorem}\label{T:1Dyckpathversion}
Let $n\geq 4$. 
Let $\pi\in L_{n,n}^+$. 
Let $w:=w_1w_2\ldots w_n$ denote the one-line expression for $\phi(\pi)$. 
Then $w$ has a unique 321 pattern if and only if $\pi$ has a unique peak at the third diagonal and no other peaks at the $r$-th diagonal for $r\geq 4$.
\end{Theorem}

\begin{proof}
$(\Rightarrow)$ Let $w \in \mathbf{S}_n^{312}$ be a permutation such that the 321 pattern occurs exactly once in $w$. 
Since a 312-avoiding permutation is also a 3412-avoiding permutation, by~\cite[Theorem 1.2]{LeeMasudaPark} that we mentioned in the preliminaries section, 
we know that a reduced expression of $w$ contains 
a braid relation of the form $s_i s_{i+1} s_i$ for some $i\in \{1,\dots, n-1\}$ as a factor only once and no other repetitions. 
Since there are no repetitions of the simple transpositions (other than $s_i$ and $s_{i+1}$), 
we cannot introduce any braid relations other than $s_i s_{i+1} s_i$ or $s_{i+1}s_is_{i+1}$ in any reduced expression of $w$. 
In other words, in any reduced expression of $w$, we will have the same braid relation $s_i s_{i+1} s_i$ (or, $s_{i+1}s_is_{i+1}$) exactly once. 

We now apply the Bandlow-Killpatrick isomorphism $\psi$ to $w$. 
Let $\pi$ denote the image $\psi(w)$. 
Then $w = \phi ( \psi(w))$. 
We write $w$ as a product of simple transpositions by using the segments $\sigma_1,\dots , \sigma_n$ as in Remark~\ref{R:itisreduced}.
If a peak of $\pi$ occurs at the $r$-th diagonal where $r\geq 4$, then we see that there are at least two simple transpositions $s_t$ and $s_{t+1}$ that appear in two consecutive segments.
We already mentioned in Remark~\ref{R:itisreduced} that the product $\sigma_1\cdots \sigma_n$ gives a reduced expression for $w$.
But this contradicts with our assumption a reduced expression of $w$ cannot contain more than one repetition of a simple transposition. 
Hence, we must have $r\leq 3$. 
Now, if $\pi$ has more than one peak at the third diagonal, then once again by looking at $w$ we find a repetition of two distinct simple transpositions in the reduced expression $\sigma_1\cdots \sigma_n$. 
Once again, we find a contradiction with~\cite[Theorem 1.2]{LeeMasudaPark}.
It follows that the Dyck path $\pi$ corresponding to $w$ cannot contain more than one peak at its third diagonal, and it has no peak at the $r$-th diagonal if $r\geq 4$.
 
$(\Leftarrow)$ We assume that $\pi$ is a Dyck path such that $\pi$ has no peak at its $r$-th diagonal if $r\geq 4$, 
and it has a unique peak on its third diagonal. 
Let $\sigma_1, \sigma_2 , \cdots , \sigma_n$ denote the segments of $\phi(\pi) = w$ as in Example~\ref{E:BK}.
Then we know that for some $i\in \{1,\dots, n\}$, $\sigma_{i+1}$ is equal to $s_i s_{i+1}$, and $\sigma_{i} = s_i$. 
All other segments of $\phi(\pi)$ contain at most one simple transposition. 
It follows that we have a 3412-avoiding permutation $w$ which contains the pattern 321 only once.
But we already know the stronger condition that $w$ is a 312-avoiding permutation since $w$ is the image of a Dyck path
under the Bandlow-Killpatrick isomorphism. 
This finishes the proof of our theorem.
\end{proof}

We finish this section by listing the Dyck paths for all partition Schubert subvarieties of the flag variety $Fl(n,\C)$.

\begin{Example}
In Figure~\ref{F:forn=5}, we depicted all Dyck paths $\pi\in L_{5,5}^+$ whose permutations contain the pattern 321 exactly once. 
These examples confirm our Theorem~\ref{T:1Dyckpathversion}. 
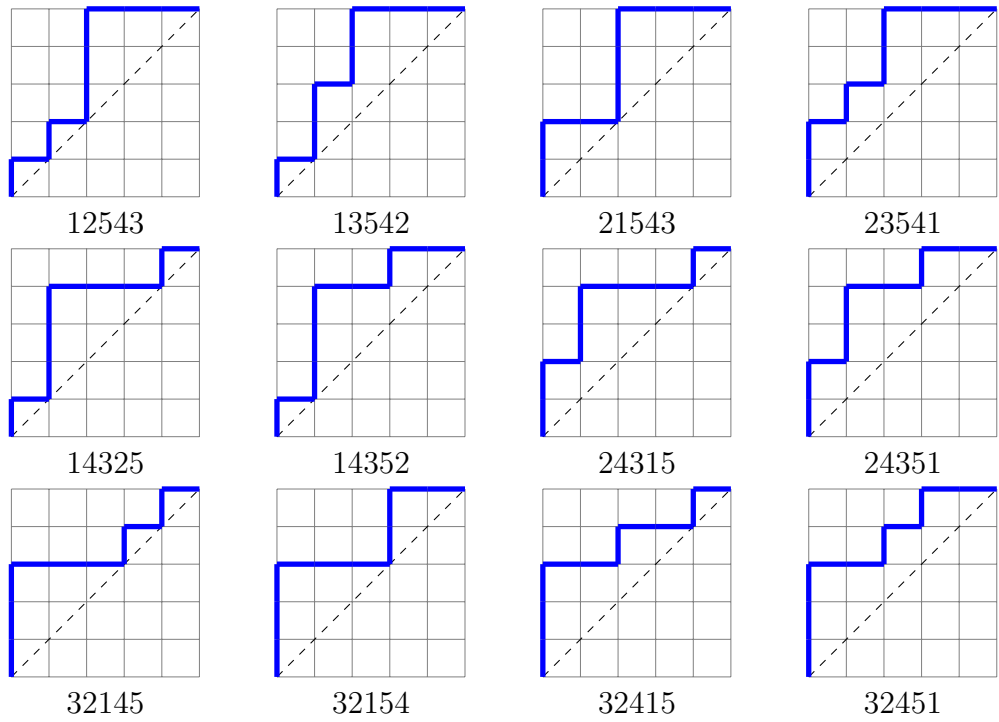
\begin{figure}[htp]
\begin{center}
 \begin{tikzpicture}[scale=0.5]
    \NEpath{0,0}{5}{5}{1,0,1,0,1,1,1,0,0,0};
    \node[below=2pt] at (2.5,0) {$12543$};
\end{tikzpicture}    
\quad\quad
\begin{tikzpicture}[scale=0.5]
   \NEpath{0,0}{5}{5}{1,0,1,1,0,1,1,0,0,0};
   \node[below=2pt] at (2.5,0) {$13542$};
\end{tikzpicture}
\quad\quad
\begin{tikzpicture}[scale=0.5]
   \NEpath{0,0}{5}{5}{1,1,0,0,1,1,1,0,0,0};
   \node[below=2pt] at (2.5,0) {$21543$};
\end{tikzpicture}
\quad \quad
 \begin{tikzpicture}[scale=0.5]
   \NEpath{0,0}{5}{5}{1,1,0,1,0,1,1,0,0,0};
   \node[below=2pt] at (2.5,0) {$23541$};
\end{tikzpicture} \\
 \begin{tikzpicture}[scale=0.5]
    \NEpath{0,0}{5}{5}{1,0,1,1,1,0,0,0,1,0};
    \node[below=2pt] at (2.5,0) {$14325$};
\end{tikzpicture}    
\quad\quad
\begin{tikzpicture}[scale=0.5]
   \NEpath{0,0}{5}{5}{1,0,1,1,1,0,0,1,0,0};
   \node[below=2pt] at (2.5,0) {$14352$};
\end{tikzpicture}
\quad\quad
\begin{tikzpicture}[scale=0.5]
   \NEpath{0,0}{5}{5}{1,1,0,1,1,0,0,0,1,0};
   \node[below=2pt] at (2.5,0) {$24315$};
\end{tikzpicture}
\quad \quad
 \begin{tikzpicture}[scale=0.5]
   \NEpath{0,0}{5}{5}{1,1,0,1,1,0,0,1,0,0};
   \node[below=2pt] at (2.5,0) {$24351$};
\end{tikzpicture} \\
 \begin{tikzpicture}[scale=0.5]
    \NEpath{0,0}{5}{5}{1,1,1,0,0,0,1,0,1,0};
     \node[below=2pt] at (2.5,0) {$32145$};
\end{tikzpicture}    
\quad\quad
\begin{tikzpicture}[scale=0.5]
   \NEpath{0,0}{5}{5}{1,1,1,0,0,0,1,1,0,0};
    \node[below=2pt] at (2.5,0) {$32154$};
\end{tikzpicture}
\quad\quad
\begin{tikzpicture}[scale=0.5]
   \NEpath{0,0}{5}{5}{1,1,1,0,0,1,0,0,1,0};
    \node[below=2pt] at (2.5,0) {$32415$};
\end{tikzpicture}
\quad \quad
 \begin{tikzpicture}[scale=0.5]
   \NEpath{0,0}{5}{5}{1,1,1,0,0,1,0,1,0,0};
    \node[below=2pt] at (2.5,0) {$32451$};
\end{tikzpicture}
\caption{The Dyck paths of permutations $w\in \mathbf{S}_5^{312}$ that contain 321 pattern only once.}
\label{F:forn=5}
\end{center}
\end{figure}
\end{Example}

Let $\mathbf{NT}_n^{312}$ denote the following subset of $\mathbf{S}_n$:
\begin{align*}
\mathbf{NT}_n^{312} := \{w\in \mathbf{S}_n^{312} :\ c_T(X_w) = 1 \ \text{ and }\ c_{B_L}(X_w) = 0 \},
\end{align*}
where $B_L$ is a Borel subgroup of the Levi subgroup of the isotropy subgroup of $X_w$ in $G$. 
By using Theorem~\ref{T:1Dyckpathversion}, we find a formula for the cardinality of $\mathbf{NT}_n^{312}$. 

\begin{Theorem}\label{T:countofNTn312}
Let $w\in \mathbf{S}_n^{312}$.
If the $T$-complexity of $X_w$ is one, then $X_w$ is a nearly toric variety. 
In other words, if $c_T(X_w) = 1$, then $X_w \in \mathbf{NT}_n^{312}$. 
Furthermore, the cardinality of $\mathbf{NT}_n^{312}$ is given by
\begin{align*}
    &|\mathbf{NT}_n^{312}|=(n-2)2^{n-3}.
\end{align*}
\end{Theorem}

\begin{proof}
Let $w\in \mathbf{S}_n^{312}$. 
Then $w$ avoids both 3412 and 4231 patterns. 
In particular, by the Lakshmibai-Sandhya criterion for smoothness, we know that the Schubert variety $X_w$ is smooth. 
At the same time, since $w$ avoids 312, it avoids 25314 pattern as well. 
Hence, by our Theorem~\ref{intro:T1} (2), we see that $X_w$ is a spherical variety. 
In other words, we have $w\in \mathbf{NT}_n^{312}$.
This finishes the proof of our first assertion. 

Next, we observe that, by Theorem~\ref{T:1Dyckpathversion}, it suffices to count the Dyck paths corresponding to the elements of $\mathbf{NT}_n^{312}$. 
Recall that we denote by $\psi$ the Bandlow-Killpatrick isomorphism.
Let $\pi \in L_{n,n}^+$, where $\pi = \psi(w)$, where $w\in \mathbf{NT}_n^{312}$.
If there is a unique peak of $\pi$ at the 3-rd diagonal, then it appears at one of the $n-2$ lattice points on the 3-rd diagonal. 
Next, we choose the peaks of $\pi$ on the 2-nd diagonal. 
Clearly, there are $n-3$ lattice to choose from. 
Therefore, our lattice path $\pi$ has total $k+1$ peaks, then it can be chosen in $(n-2) {n-3\choose k}$ different ways. 
It follows that the cardinality of $\mathbf{NT}_n^{312}$ is given by
\begin{align*}
  |\mathbf{NT}_n^{312}|=\sum_{k=0}^{n}\binom{n-3}{k}\cdot (n-2)=(n-2)2^{n-3}.
\end{align*}
This finishes the proof of our corollary.
\end{proof}

\section{Spherical Dyck Paths}\label{S:SphericalDyck}\label{S:Allpartition}

We recall our notion of a spherical Dyck path in more detail.


Let $\tau$ be a Dyck path in $L_{n,n}^+$. 
We call the lattice path $\tau'$ that is obtained from $\tau$ by adding an East step in front of its word the {\em $E$ extension of $\tau$}. 
In other words, if the word of $\tau$ is given by $a_1 a_2\ldots a_r$, where $a_i \in \{N,E\}$ for $i\in \{1,\dots, r\}$,
then the word of the $E$ extension of $\tau$ is given by $\tau':= a_0a_1a_2\ldots a_r$, where $a_0 = E$.

\begin{Definition}\label{D:connectedcomponent}
Let $r\in \{0,1,\dots, n-1\}$. 
Let $\pi \in L_{n,n}^+$. 
A subpath of $\pi^{(r)}$ is said to be a {\em connected component of $\pi^{(r)}$} if it starts and ends at the $r$-th diagonal,
and it touches the $r$-th diagonal exactly twice. 
We call a Dyck path $\pi$ a {\em spherical Dyck path} if every connected component $\tau$ of $\pi$ is either an elbow or a ledge, or if every connected component of $\tau$ that lies above the line $y-x- 1=0$ is either an elbow or a ledge that extends to the main diagonal of $\pi$. 
\end{Definition}

Let us explain this definition on a figure. 
A Dyck path $\pi$ is spherical if every connected component on the first diagonal $\pi^{(0)}$ is either an elbow or a ledge as depicted in Figure \ref{fig:1}(a), or every connected component of $\pi^{(1)}$ is an elbow, or a ledge whose $E$  extension is the initial step of a connected component of $\pi^{(0)}$ as depicted in Figure \ref{fig:1}(b). 
\begin{center}
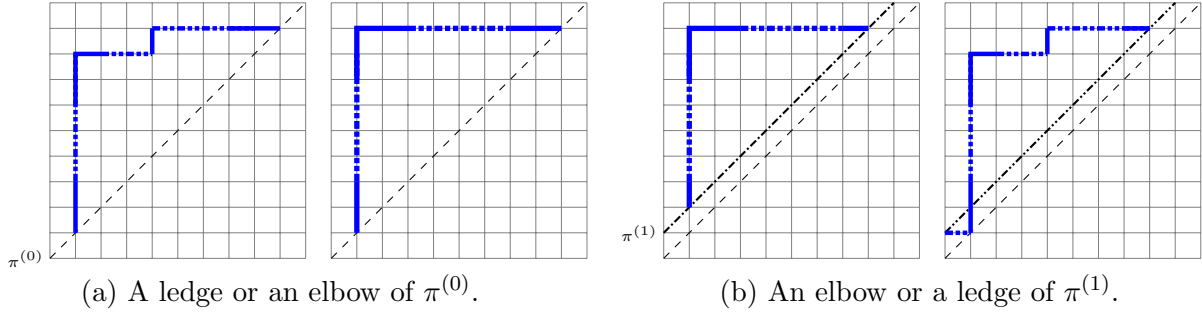
\begin{figure}[htp]
 \begin{tikzpicture}[scale=.34]
\NEpath{0,0}{10}{10}{};
    \draw[line width=1.7pt,blue,densely dash dot dot] (1,3) -- +(0,5);
    \draw[line width=1.7pt,blue,densely dash dot dot] (4,9) -- +(4,0);
    \draw[line width=1.7pt,blue,densely dash dot dot] (2,8) -- +(2,0);
    \draw[-,line width=1.7pt,blue] (1,6) -- (1,8) -- (2,8); 
     \draw[-,line width=1.7pt,blue] (1,1) -- (1,3); 
     \draw[-,line width=1.7pt,blue] (7,9) -- (9,9); 
       \draw[-,line width=1.7pt,blue] (4,8) -- (4,9); 
    \node[below=2pt] at (9,0) {\small (a) A ledge or an elbow of $\pi^{(0)}$.};
    \node[left] at (0.2,0) {\tiny$\pi^{(0)}$};

   \begin{scope}[xshift=11cm]
     \NEpath{0,0}{10}{10}{};
    \draw[line width=2pt,blue,densely dash dot dot] (1,3) -- +(0,5);
    \draw[line width=2pt,blue,densely dash dot dot] (3,9) -- +(4,0);
    \draw[-,line width=2pt,blue] (1,7) -- (1,9) -- (3,9); 
     \draw[-,line width=2pt,blue] (1,1) -- (1,3); 
     \draw[-,line width=2pt,blue] (7,9) -- (9,9); 
    \end{scope}
    
     \begin{scope}[xshift=24cm]
     \NEpath{0,0}{10}{10}{};
    \draw[thick,densely dash dot] (0,1) -- +(9,9); 
    \draw[line width=2pt,blue,densely dash dot dot] (1,3) -- +(0,5);
    \draw[line width=2pt,blue,densely dash dot dot] (3,9) -- +(4,0);
    \draw[-,line width=2pt,blue] (1,7) -- (1,9) -- (3,9); 
     \draw[-,line width=2pt,blue] (1,2) -- (1,3); 
     \draw[-,line width=2pt,blue] (7,9) -- (8,9); 
    \node[below=2pt] at (10,0) {\small (b) An elbow or a ledge of $\pi^{(1)}$.};
    \node[left] at (0.2,1) {\tiny$\pi^{(1)}$};
    \end{scope}

     \begin{scope}[xshift=35cm]
     \NEpath{0,0}{10}{10}{};
\draw[line width=1.7pt,blue,densely dash dot dot] (1,3) -- +(0,5);
\draw[thick,densely dash dot dot,] (0,1) -- +(9,9);
    \draw[line width=1.7pt,blue,densely dash dot dot] (4,9) -- +(4,0);
    \draw[line width=1.7pt,blue,densely dash dot dot] (2,8) -- +(2,0);
    \draw[-,line width=1.7pt,blue] (1,6) -- (1,8) -- (2,8); 
     
     \draw[-,line width=1.7pt,blue, densely dash dot dot] (0,1)--(1,1); 
     
       \draw[-,line width=1.7pt,blue] (1,1) -- (1,3); 
     \draw[-,line width=1.7pt,blue] (7,9) -- (8,9); 
       \draw[-,line width=1.7pt,blue] (4,8) -- (4,9); 
    \end{scope}
\end{tikzpicture} 
\caption{Spherical Dyck paths}
\label{fig:1}
\end{figure}
\end{center}
\medskip

\begin{Remark}
It is important to realize that, although its $E$ extension does not fit into the $n\times n$ grid that contains the elements of $L_{n,n}^+$, the ledge depicted in Fig~\ref{fig:00} is a valid ledge in $\pi^{(1)}$ for some spherical Dyck path $\pi$.
We call such a ledge an {\em initial ledge}. 
\begin{figure}[htp]
\begin{center}
\begin{tikzpicture}[scale=.34]
\NEpath{0,0}{10}{10}{};
\draw[thick,densely dash dot dot,] (0,1) -- +(9,9);

\draw[line width=1.7pt,blue,densely dash dot dot] (0,3) -- +(0,5);
\draw[line width=1.7pt,blue,densely dash dot dot] (1,8) -- +(2,0);
\draw[-,line width=1.7pt,blue] (0,5) -- (0,8) -- (1,8); 
\draw[-,line width=1.7pt,blue] (3,8) -- (4,8) -- (4,9)--(5,9); 
\draw[line width=1.7pt,blue,densely dash dot dot] (5,9) -- +(3,0);
\draw[-,line width=1.7pt,blue] (0,0) -- (0,3); 

\draw[line width=1.7pt,blue,densely dash dot dot] (-1,0) -- +(1,0);

\end{tikzpicture} 
\caption{A ledge that can appear in a spherical Dyck path.}
\label{fig:00}
\end{center}
\end{figure}
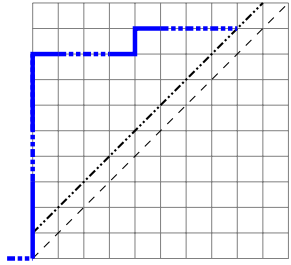
\end{Remark}

In our next lemma, we determine the complete list of non-spherical Dyck paths of size $5$.

\begin{Lemma}\label{L:exceptional}
Let $\pi\in L_{5,5}^+$. 
Let $w_\pi$ denote the permutation defined by $w_\pi:=\phi(\pi)$, where $\phi$ is the inverse of the Bandlow-Killpatrick isomorphism. 
Then $X_{w_\pi B}$ is a spherical Schubert variety if and only if 
$\pi$ is not one of the three Dyck paths that are listed in Figure~\ref{F:exceptional}.
\begin{figure}[htp]
\begin{center}
\scalebox{0.65}{
\begin{tikzpicture}
\begin{scope}[xshift=-8cm]
\draw (0,0) grid (5,5);
\draw[color=blue, line width=3pt] (0,0) -- (0,3) -- (1,3) -- (1,5) -- (5,5);
\draw[ultra thick, dashed] (0,0) -- (5,5);
\draw[ultra thick, dashed ] (0,1) -- (4,5);
\end{scope}

\begin{scope}[xshift=0cm]
\draw (0,0) grid (5,5);
\draw[color=blue, line width=3pt] (0,0) -- (0,2) -- (1,2) -- (1,4) -- (2,4) -- (2,5) -- (5,5);
\draw[ultra thick, dashed] (0,0) -- (5,5);
\draw[ultra thick, dashed ] (0,1) -- (4,5);
\end{scope}

\begin{scope}[xshift=8cm]
\draw (0,0) grid (5,5);
\draw[color=blue, line width=3pt] (0,0) -- (0,3) -- (1,3) -- (1,4) -- (2,4)
--(2,5) --(5,5);
\draw[ultra thick, dashed] (0,0) -- (5,5);
\draw[ultra thick, dashed] (0,1) -- (4,5);
\end{scope}
\end{tikzpicture}    
}
\end{center}
\caption{The non-spherical Dyck paths of size 5.}
\label{F:exceptional}
\end{figure}
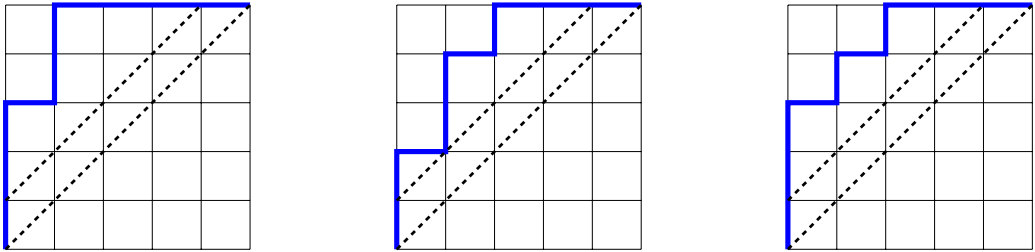
\end{Lemma}

\begin{proof}
Let $a,b$, and $c$ denote the Dyck paths that are depicted in Figure~\ref{F:exceptional} from left to right.
By direct computation, we see that the permutation corresponding to $a$ is given by 35421.
The permutation corresponding to $b$ is given by 24531. 
The permutation corresponding to $c$ is given by 34521.
All of these permutations are among the patterns in the set (\ref{A:Gaetz}).
Hence, their Schubert varieties cannot be spherical. 
We checked directly by (\ref{A:Gaetz}) that 
for every $\pi\in L_{5,5}^+ \setminus \{a,b,c\}$ the Schubert variety $X_{w_\pi B}$
is a spherical Schubert variety in the flag variety $Fl(5,\C)$. 
This finishes the proof of our assertion. 
\end{proof}

We proceed with a simple observation. 

\begin{Proposition}\label{P:factorizationuv}
Let $w\in \mathbf{S}_n$.
We assume that $w$ has a factorization of the form $w=uv$, where $u$ and $v$ are permutations such that $uv=vu$. 
Then $X_w$ is a spherical Schubert variety if and only if both of $X_u$ and $X_v$ are spherical Schubert varieties. 
\end{Proposition}
\begin{proof}
Without loss of generality we may assume that $u$ (resp. $v$) belong to the subgroup
$W_r:=\langle s_1,\dots, s_{r-1} \rangle $ (resp. the subgroup $W_{n-r}:=\langle s_r,\dots, s_{n-1} \rangle $) of $\mathbf{S}_n$.
Hence, the standard Levi subgroup $L$ of the isotropy subgroup of $X_w$ is a subgroup of 
$\mathbf{GL}(r,\C)\times \mathbf{GL}(n-r,\C)$. In particular it has a product decomposition of the form $L_1\times L_2$, where $L_1$ is a standard Levi subgroup of $\mathbf{GL}(r,\C)$ and $L_2$ is a standard Levi subgroup of $\mathbf{GL}(n-r,\C)$.
Let $L_u$ denote the standard Levi subgroup of $\mathbf{GL}(n,\C)$ consisting of block matrices of the form 
$\begin{bmatrix} A & 0 \\ 0 & B \end{bmatrix}$, where $A\in L_1$ and $B$ is an element of the maximal diagonal torus of $\mathbf{GL}_{n-r}(\C)$.
Likewise, let $L_v$ denote the standard Levi subgroup of $\mathbf{GL}(n,\C)$ consisting of block matrices of the form 
$\begin{bmatrix} A & 0 \\ 0 & B\end{bmatrix}$, where $B\in L_2$ and $A$ is an element of the maximal diagonal torus of $\mathbf{GL}(r,\C)$.
It is easy to check that $L_u$ (resp. $L_v$) is the standard Levi subgroup of the stabilizer of $X_u$ (resp. of $X_v$) in $\mathbf{GL}(n,\C)$. Notice that both of $L_u$ and $L_v$ are Levi subgroups of $L$. 
It is now evident that $L$ acts spherically on $X_w$ if and only if $L_u$ (resp. $L_v$) acts spherically on $X_u$ (resp. on $X_v$). This finishes the proof of our assertion.
\end{proof}

\begin{Lemma}\label{L:edgeorledge}
Let $w\in \mathbf{S}_n^{312}$. 
Let $\pi$ denote the Dyck path such that $\phi(\pi)=w$, where $\phi$ is the inverse of the Bandlow-Killpatrick map.
If $\pi$ is a spherical Dyck path, then $X_w$ is a spherical variety. 
\end{Lemma}
\begin{proof}
Recall that, to find $\phi(\pi)$, a reduced expression is found by reading the simple transpositions in the boxes under the path and above the main diagonal from left to right, top to bottom. 
Let $A_1,\dots, A_r$ (resp. $M_1,\dots, M_s$) denote the irreducible components of $\pi^{(0)}$ (resp. of $\pi^{(1)}$) listed from bottom to top (resp. bottom to top). 
Then each $A_i$ ($i\in \{1,\dots, r\}$) gives us a factor $u_i$, a certain product of simple transpositions, of $w$.
It is easy to see that, for every $1\leq i<j\leq r$, the factors $u_i$ and $u_j$ commute since they do not share any simple transpositions in their reduced expressions. 
This means that we can apply Proposition~\ref{P:factorizationuv} to reduce the proof of our claim to a single connected component.
Therefore, we proceed with the assumption that $\pi$ is one of the following three type of paths: 
1) a ledge, 
2) an elbow, or, 
3) a Dyck path that touches the main diagonal exactly twice and for each $i\in \{2,\dots, s\}$, the connected component $M_i$ of $\pi^{(1)}$ is an elbow, but $M_1$ is either an elbow or a ledge. 
Notice that if $\pi$ is as in 1) or 2), then it is already as in 3).
For this reason, we assume that $\pi$ is a Dyck path of the last kind. 
Let $w_1\ldots w_n$ denote the one-line expression of $w$. 
Thanks to the specific nature of $\pi$, we have a precise description of the ``ascents'' of $w$.
Here, by an {\em ascent of $w$}, we mean an index $i \in \{1,\dots,n-1\}$ such that $w_i < w_{i+1}$. 
It follows from Lemma~\ref{L:mds} that, since $\pi$ is a Dyck path of type 3, all ascents of $w$, except possibly one of them, occur at the indices where the connected components $M_2,\dots,M_s$ touch the first diagonal. 
The first connected component, that is $M_1$, gives an additional ascent if and only if it is a ledge.
Since $M_2,\dots, M_s$ are elbows, they all look like $NN\dots N EE\dots E$. 
Let us denote by $m_i$ ($i\in \{1,\dots, r\}$) the total number of $N$ steps in the connected component $M_i$ ($i\in \{1,\dots, r\}$). 
Then we have $m_1 +\cdots + m_s = n-1$. 
We will determine the one-line expression of $w$ by using the numbers $m_1,\dots, m_s$. 
But first we observe that by the connectedness of $\pi^{(0)}$ and Lemma~\ref{L:Helpful1} we know that $w_n =1$. 
Since the ascents of $w$ are determined by the lattice points where $\pi$ touches the first main diagonal, we have the following equalities of 
ordered sets: 
\begin{eqnarray*}
\{w_{m_1-1},\dots, w_1\} &=& \{2,3,\dots, m_1\}\\
\{w_{m_1+m_2-1},\dots, w_{m_1}\} &=& \{m_1+1, m_1+2,\dots, m_1+m_2\}\\
&\vdots & \\
\{w_n,\dots, w_{m_1+\cdots m_{s-1}}\} &=& \{1, m_1+\dots +m_{s-1}+1,\dots,  m_1+\dots+m_s\}.
\end{eqnarray*}
Here, by an ordered set, we mean a set whose elements are listed in an increasing order from left to right. 
It is now easy to check that the permutation $w_1\ldots w_n$ avoids the 312 pattern. 
Thus, $w$ avoids all 19 permutations listed in Figure~\ref{F:fig:3}.
In order for $X_w$ to be a spherical Schubert variety, $w$ has to avoid the remaining patterns $24531$, $34521$, and $35421$. 
Notice that in each of these patterns, the pattern 231 occurs at least twice, where one them occurs in the first four digits. 
However, in our permutation $w_1\ldots w_n$, all 231 patterns are of the form $w_i < w_j$, $w_n < w_i$, where $1\leq i<j<n$. 
This means that none of the patterns $24531$, $34521$, and $35421$ can occur in $w_1\ldots w_n$. 
Hence, we showed that if $\pi$ is a spherical Dyck path satisfying the condition 3), then the corresponding permutation $w$ avoids all 21 patterns of the theorem of Gaetz that we mentioned in the preliminaries section.
In other words, $X_w$ is a spherical Schubert variety. 
This finishes the proof of our lemma.
\end{proof}

In our next lemma, we will analyze the left descent sets of certain elements of $\mathbf{S}_n^{312}$. 
\begin{Lemma}\label{L:skippedsr}
Let $w\in \mathbf{S}_n^{312}$.
We assume that the Dyck path $\pi$ of $w$ has the property that $\pi^{(0)}=\pi$.
In other words, we assume that $\pi^{(0)}$ has only one connected component. 
If the set of lattice points where $\pi^{(1)}$ touches the first main diagonal is given by $\{(i_1,i_1+1)=(0,1),(i_2,i_2+1),\dots, (i_r,i_r+1)=(n-1,n)\}$, then the following simple reflections cannot be contained in the left descent set of $w$:
\[
s_{i_{r-1}}, s_{i_{r-2}},\dots, s_{i_{3}}.
\]
\end{Lemma}

Before we present the proof of our lemma, we give an example to convey the idea behind its proof. 

\begin{Example}
Let $\pi$ denote the Dyck path that is depicted in Figure~\ref{F:descentset}. Then we have $\pi = \pi^{(0)}$.

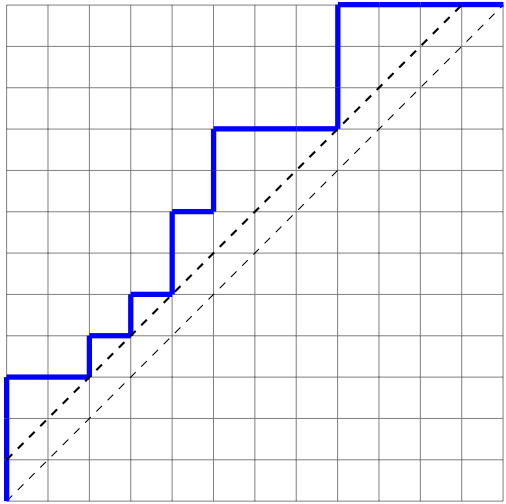
\begin{figure}[htp]
\begin{center}
 \begin{tikzpicture}[scale=0.55]
   \NEpath{0,0}{12}{12}{1,1,1,0,0,1,0,1,0,1,1,0,1,1,0,0,0,1,1,1,0,0,0,0};
\draw[thick, dashed ] (0,1) -- (11,12);
\end{tikzpicture}
\caption{The Dyck path of $w= 3\ 2\ 4\ 5\ 7\ 9\ 8\ 6\ 12\ 11\ 10\ 1$.}
\label{F:descentset}
\end{center}
\end{figure}
The list of lattice points where $\pi$ crosses the first main diagonal is 
\[
(0,1), (2,3), (3,4), (4,5), (8,9),(11,12).
\]
Then, the relevant indices are given by
\[
(i_1,i_2,\dots, i_r) = (i_1,\dots, i_6) = (0,2,3,4,8,11).
\]
A reduced expression of $w$ is given by 
\begin{align}\label{A:inverseBKmap}
w = \underbrace{s_{9} s_{10} s_{11} | s_9 s_{10} | s_9 }| \underbrace{s_6s_7s_8 | s_6s_7} | \underbrace{ s_5s_6 | s_5} | \underbrace{s_4} |\underbrace{ s_3}| \underbrace{ s_1s_2 |s_1}.
\end{align}
It is easy to verify that the left descent set of $w$ is given by $\{s_1,s_6,s_7,s_9,s_{10},s_{11}\}$.
In particular, the simple transpositions $s_3,s_4,s_8$ are excluded from the left descent set of $w$. 
Notice that in (\ref{A:inverseBKmap}) the first occurrence of $s_3$ is after $s_4$, the first occurrence of $s_4$ is after $s_5$, and the first occurrence of $s_8$ is after $s_9$. For these simple transpositions, there are no braid relations in the reduced expression (\ref{A:inverseBKmap}) that might be used for reversing their orders. 
\end{Example}

\begin{proof}[Proof of Lemma~\ref{L:skippedsr}]

Recall that, to determine the permutation $w$ from its Dyck path $\pi$, we apply the inverse of the Bandlow-Killpatrick map, building a reduced expression by reading the simple transpositions in the boxes under the path and above the main diagonal from left to right, top to bottom. 
Now, if $\pi$ passes through the lattice point $(i,i+1)$ for some $i\in \Z_+$, then in the reduced expression of $w$, $s_i$ appears after (that is to say, to the left) of all occurrences of $s_{i+1}$. 
Consequently, in the one-line expression $w_1w_2\ldots w_n$ of $w$, we have $w_i < w_{i+1}$. 
According to our convention (which follows the convention of Bandlow and Killpatrick), the multiplication by $s_i$ on the left interchanges the entries $w_i$ and $w_{i+1}$. It follows that $\ell(s_i w) > \ell(w)$. 
Therefore, the simple reflection $s_i$ cannot be a left descent for $w$. 
This finishes the proof of our assertion.
\end{proof}

The final lemma that we want to prove before proving our main theorem is about the one-line expression of $w\in \mathbf{S}_n^{312}$. 
\begin{Definition}
Let $w_1\ldots w_n$ be the one-line expression for an element $w\in \mathbf{S}_n$. 
A {\em maximal decreasing subsequence} (an {\em mds} for short) of $w_1\ldots w_n$ is a subsequence $w_i w_{i+1}\dots w_j$ with consecutive indices such that  
$w_j  < w_{j-1}< \dots < w_{i+1} < w_i$ and there is no other such subsequence of $w_1\ldots w_n$ having $w_i w_{i+1}\dots w_j$ as a proper subsequence. 
\end{Definition}

For example, for $w= 3\ 2\ 4\ 5\ 7\ 9\ 8\ 6\ 12\ 11\ 10\ 1$, the mds's are given by $3\ 2$, 4, 5, 7, $9\ 8\ 6$, and $12\ 11\ 10\ 1$.
\begin{Lemma}\label{L:mds}
Let $w_1\ldots w_n$ be the one-line expression for an element $w\in \mathbf{S}_n^{312}$. 
Let $\pi$ denote the Dyck path such that $\phi(\pi)=w$, where $\phi$ is the inverse of the Bandlow-Killpatrick map.
Then the heights of the East steps of $\pi$ encode the first entries of the mds's of $w_1\ldots w_n$.
\end{Lemma}
\begin{proof}
The proof will follow from the fact that $\phi(\pi)$ is a product of segments, which are products of simple transpositions whose indices increase from left to right, and $w_1\ldots w_n$ is obtained by evaluating this product of segments on the left.

First we organize consecutive segments in $\phi(\pi)$ into groups, where each group consists of segments which are nested
(see the example (\ref{A:inverseBKmap}), where we indicated these groups by using under-braces).
If a segment group begins with $s_k$ (for some $k\in [n-1]$), then the evaluation of $k$ at this segment group gives a number that is strictly greater than $k$. 
If $k$ is the index of a simple transposition that appears in the middle of a segment, then the evaluation of $k$ in this group gives a number that is less than $k$. 
In other words, by evaluating a segment group that starts with $s_k$ at $k$, we obtained the maximum element of an mds in $w_1\ldots w_n$. 
Furthermore, since the indices of the initial terms of each segment block form a decreasing sequence (because of the reading order from left to right, top to bottom)
we see that the each initial index in a segment corresponds to the maximal entry of an mds in $w_1\ldots w_n$. 
This finishes the proof of our assertion.
\end{proof}

We are now ready to close our paper by proving the last theorem of our paper.
We state it here once more for the convenience of the reader.

\begin{Theorem}\label{T:sphericalDyckpaths}
Let $w\in \mathbf{S}_n^{312}$. 
Let $\pi$ denote the Dyck path of size $n$ corresponding to $w$.
Then $X_w$ is a spherical Schubert variety if and only if $\pi$ is a spherical Dyck path. 
\end{Theorem}

\begin{proof}
$(\Leftarrow)$ This implication follows from Lemma~\ref{L:edgeorledge}. 
\medskip

$(\Rightarrow)$ We prove our claim by using induction.
The base case is when $n=5$. 
In this case, our claim follows from Lemma~\ref{L:exceptional}. 
We proceed with the assumption that our claim holds true for every $v\in \mathbf{S}_{n-1}^{312}$ such that $X_v$ is a spherical Schubert variety in $Fl(n-1,\C)$. 
Let $w\in \mathbf{S}_n^{312}$ be a permutation such that $X_w$ is a spherical Schubert variety in $Fl(n,\C)$.
We now notice a simple reduction argument.
Let $\pi:=\pi_w$ denote the Dyck path of $w$. 
If $w$ has a factorization of the form $w=uv$, where $u$ and $v$ are (nontrivial) permutations that commute with each other, then $\pi^{(0)}$ has at least two connected components whose words begin with $NN\ldots$. 
Conversely, if $\pi^{(0)}$ has more than one connected components, where at least two of them begins with $NN\ldots$, then 
the corresponding permutation has a factorization into commuting permutations.  
Thus, in light of Proposition~\ref{P:factorizationuv}, thanks to our induction hypothesis, it suffices to focus on the permutations such that $\pi_w^{(0)}$ has a single connected component. 
Therefore, we proceed with the assumption that $\pi = \pi^{(0)}$.
Let us assume that $\pi$ is not a spherical Dyck path.
(Otherwise, there is nothing to prove.) 
Note that, since we have $\pi = \pi^{(0)}$, the Dyck path system $\pi^{(1)}$ is in fact a single Dyck path of size $n-1$.

Recall from~\cite{CanSaha} (or from~\cite{GHY2024}) that $X_w$ is a spherical Schubert variety if and only if $w$ is of the form $w= w_0(J(w)) c$, where $J(w)$ is the left descent set of $w$. 
Therefore, $w$ might be a Coxeter element of some parabolic subgroup of $\mathbf{S}_n$. 
Let us assume that this is the case. 
Then every reduced expression for $w$ has no repeated factors in it. 
In this case, the Dyck path of $w$ has no segments that have more than one simple transposition in it. 
This is equivalent to the statement that the connected components of $\pi^{(2)}$ are isolated lattice points. 
Then we see that $\pi$ has the form $NNENEN\cdots ENEE$.
It follows that every connected component of $\pi^{(1)}$ is an elbow.
In this case, we see that $\pi$ is a spherical Dyck path. 
Now let us assume that $w$ is $w_0(J(w))$, that is, $c= id$. 
In this case, it is easy to check that $\pi$ is an elbow. 
Hence, it is a spherical Dyck path. 
\medskip

We proceed with the assumption that $w$ is neither a standard Coxeter element of any parabolic subgroup of $W$ nor the longest element of $J(w)$. Thus, $w= w_0(J(w))c$, where $w_0(J(w))\neq id$ and $c\neq id$. 
We list the connected components of $\pi^{(1)}$ from bottom to top as in $M_1,\dots, M_s$. 
Let us first assume that $s >1$ and that $M_1$ is an elbow of size $k \geq 1$. 
By applying the inverse of the Bandlow-Killpatrick map, we see that a reduced expression of $w$ starts with a reduced expression of the longest element of the parabolic subgroup of $\mathbf{S}_n$ generated by $s_{n-1},s_{n-2},\dots, s_{n-k}$.
Furthermore, in this case, by Lemma~\ref{L:skippedsr}, we know that the simple reflection $s_{n-k-1}$ is not in the left descent set of $w$.
This means that $w_0(J(w))$ is a product of the form $w_0(I_1)w_0(I_2)$, where $I_1= \{s_{n-1},s_{n-2},\dots, s_{n-k}\}$ and $I_2 = J(w) \setminus I_1$. Since $w$ is given by $w_0(J(w))c$, where $c$ is a Coxeter element of some parabolic subgroup of $W$,
we see that $w_0(I_1)w = w_0(I_2)c$. 
In other words, the element $w_0(I_1)w \in W$ is a spherical element, that is to say, $X_{w_0(I_1)w}$ is a spherical Schubert variety in $Fl(n,\C)$. 
At the same time, $w_0(I_1)w$ is an element of $\mathbf{S}_n^{312}$.
This follows from the fact that the reduced expression of $w_0(I_1)w$ is obtained from the reduced expression of $w$ by deleting
the simple reflections that corresponds to the lattice points under the path $M_1$ above the main diagonal. 
Now, since a reduced expression of $w_0(I_1)w$ is a product of the first $n-k-1$ simple reflections, the corresponding Schubert variety is isomorphic to a Schubert variety in $Fl(n-k,\C)$.
Then it follows from our induction hypothesis that the corresponding Dyck path $\pi'$ in $L_{n-k,n-k}^+$ is a spherical Dyck path.
But the Dyck path $\pi$ is obtained from $\pi'$ by appending the elbow of size $k$ to it. 
Hence, $\pi$ is a spherical Dyck path as well. 
\medskip

We now proceed with the assumption that $s>1$ but $M_1$ is not an elbow. 
We will use an argument that is very similar to the one we had in the previous paragraph. 
By applying the inverse of the Bandlow-Killpatrick map, we see that a reduced expression of $w$ starts with a reduced expression of an element $u$ of the parabolic subgroup of $\mathbf{S}_n$ generated by the elements of $I_1:=\{ s_{n-1},s_{n-2},\dots, s_{n-k}\}$.
By Lemma~\ref{L:skippedsr}, we know that the simple reflection $s_{n-k-1}$ is not in the left descent set of $w$.
This means that $\ell(u^{-1}w)=\ell(w) - \ell(u^{-1})$. 
Since every simple reflection that appears in the reduced expression of $u$ is contained in $J(w)$ as well,
we see that $\ell(u^{-1}w_0(J(w))) = \ell(w_0(J(w)))-\ell(u^{-1})$. 
Let $v$ denote $u^{-1}w_0(J(w))$. 
Then the reduced expression of $vc$ is an element of the parabolic subgroup generated by $s_1,\dots, s_{n-k-1}$.
In particular, we see that the reduced expression of $c$ does not contain any simple reflections from $I_1$. 
Likewise, the reduced expression of $v$ does not contain any simple reflections from $I_1$. 
Since $s_{n-k-1}$ is not in $J(w)$, we see that $u$ and $v$ commute with each other. 
In other words, the longest permutation $w_0(J(w))$ is a product of two commuting elements $u$ and $v$ which do not have any common simple reflection in their reduced expressions. This means that $u$ and $v$ are separately the longest elements of the Young subgroups generated by the elements of $J(w) \cap I_1$ and $J(w)\setminus I_1$, respectively. 
However, this contradicts with our assumption that $M_1$ is not an elbow. 
\medskip

It remains to show that, if $s=1$, that is to say $M_1$ is the only connected component of $\pi^{(1)}$, then $\pi$ is a spherical Dyck path. 
Towards a contradiction, we assume that $\pi$ is not a spherical Dyck path. 
We will analyze the one-line expression $w_1\ldots w_n$ of $w$. 
Under our assumptions, Lemma~\ref{L:Helpful1} implies that $w_n$ is $1$.
Likewise, Lemma~\ref{L:Helpful2} implies that $w_{n-1} = 2$. 
Thus, our assumptions imply that the number $t$ of East steps of $\pi$ on the line $y=n$ is at least 3, at most $n-2$.
In terms of one-line expression, this means that 
\[
w_{n-2} > w_{n-1}=2 > w_n = 1.
\]
Now, our assumptions imply also that $\pi$ is neither an elbow nor a ledge.
Therefore, its word will have at least two occurrences of the consecutive pattern $ENE$.
Then, by Lemma~\ref{L:mds}, the height of the second East step in each of these $ENE$'s corresponds to the first entry of an mds in $w_1\ldots w_n$. 
It follows that if we denote by $k\in [n]$ the index such that $w_k = n$, then the second $E$ of the last $ENE$ is the $k$-th East step in $\pi$. 
The two $E$'s in the preceding $ENE$ pattern correspond to the entries $w_i$ and $w_{i+1}$ of $w_1\ldots w_n$ (for some $i\in \{1,\dots, k-2\}$)
and they are ordered by $w_i < w_{i+1}$. 
(This follows from the fact that the second East step of $ENE$ is the maximal entry of the corresponding mds,
and the first East step of $ENE$ is the lowest entry of the previous mds.)
Then, we see the existence of the following patterns in $w_1\ldots w_n$: 
\[
w_i < w_{i+1} < w_k = n \quad\text{and}\quad w_{n-1}=2 > w_n=1. 
\]
But this is precisely an occurrence of the 34521 pattern in $w_1\ldots w_n$. 
Hence, by the pattern avoidance characterization of Schubert varieties, $X_w$ is not a spherical Schubert variety. 
This contradicts our initial assumption.
\end{proof}

We close our paper by mentioning our future work.
From a representation theory viewpoint, isolating combinatorially significant subfamilies among Schubert varieties corresponds to distinguishing certain Demazure modules based on their combinatorial properties.
In an upcoming paper, various combinatorial representation-theoretic properties of nearly toric Schubert varieties will be presented. 

\section*{Conflict of interest}

The authors have no conflicts of interest to declare that are relevant to the content of this article.

\section*{Acknowledgement}
Th authors gratefully acknowledges partial support from the Louisiana Board of Regents grant, contract no. LEQSF(2023-25)-RD-A-21.

\bibliography{References}
\bibliographystyle{plain}

\end{document}